%% file: BraneQuantizationPaper.tex
\documentclass[11pt]{article}
\usepackage[utf8]{inputenc}
\input{preamble}

\newcommand{\OO}{\mathcal{O}}

\newcommand{\JJ}{\mathbb J}

\newcommand{\BB}{\mathcal{B}}

\newcommand{\Ll}{\mathcal{L}}

\newcommand{\delbar}{\overline\partial}

\newcommand{\TT}{\mathbb{T}}

\newcommand{\ol}{\overline}

\renewcommand{\Re}{\mathrm{Re}}
\renewcommand{\Im}{\mathrm{Im}}

    \catcode`\_=11\relax
    
    \def\_email#1@#2\q_nil{%
      \href{mailto:#1@#2}{{\emailfont #1@#2}}
    }
    \newcommand\emailfont{\sffamily}
    \newcommand\emailampersat{{\color{red}\small@}}
    \catcode`\_=8\relax    

\usepackage{tikz}
\usetikzlibrary{matrix,arrows,decorations.pathmorphing}

\title{\Large Brane quantization of toric Poisson varieties}
\author{ Francis Bischoff\footnote{University of Oxford;  francis.bischoff@maths.ox.ac.uk}\ \ \and\ \ Marco Gualtieri\footnote{University of Toronto;  mgualt@math.toronto.edu}}
\date{}
\begin{document}
\maketitle
\vspace{-2em}
\abstract{
In this paper we propose a noncommutative generalization of the relationship between compact K\"ahler manifolds and complex projective algebraic varieties. Beginning with a prequantized K\"ahler structure, we use a holomorphic Poisson tensor to deform the underlying complex structure into a generalized complex structure, such that the prequantum line bundle and its tensor powers deform to a sequence of generalized complex branes.  Taking homomorphisms between the resulting branes, we obtain a noncommutative deformation of the homogeneous coordinate ring.  As a proof of concept, this is implemented for all compact toric K\"ahler manifolds equipped with an R-matrix holomorphic Poisson structure, resulting in what could be called noncommutative toric varieties.
 
To define the homomorphisms between generalized complex branes, we propose a  method which involves lifting each pair of generalized complex branes to a single coisotropic A-brane in the real symplectic groupoid of the underlying Poisson structure, and compute morphisms in the A-model between the Lagrangian identity bisection and the lifted coisotropic brane.  This is done with the use of a multiplicative holomorphic Lagrangian polarization of the groupoid.    
}
\tableofcontents

\section{Introduction}

   A symplectic form $\omega$ on a manifold $M$, if it has integral periods, may be prequantized to a Hermitian line bundle $L$ with unitary connection $\nabla$ whose curvature is $-2\pi i\omega$.  A compatible complex structure $I$ defines a K\"ahler structure on $M$ and, from the point of view of geometric quantization, serves as a \emph{complex polarization}, endowing $L$ with a holomorphic structure $\delbar_L = \nabla^{0,1}$ and defining a sheaf of \emph{polarized sections}: the holomorphic sections of $L$.  The geometric quantization of the symplectic manifold $(M,\omega)$ is then taken to be the space of global sections $H^0(M,L)$ of this sheaf.  The extent to which the quantization is independent of the polarization has been much studied, e.g.~\cite{Hitchin:1990gq,Axelrod:1989xt,MR2928087,MR1764435}; while the physical expectation of independence has been shown in important cases, a complete understanding of this phenomenon remains elusive. 
On the other hand, if we consider the quantizations of all multiples of our symplectic form $\omega, 2\omega, 3\omega,\ldots$, it is clear that this family of quantizations fully encodes the information of the complex polarization.  Indeed, by the Kodaira embedding theorem, the graded commutative \emph{homogeneous coordinate ring}
\begin{equation*}
A = \bigoplus_{n\geq 0} H^0(M,L^{\otimes n})
\end{equation*}
expresses the complex manifold $(M,I)$ as an algebraic subvariety of complex projective space. 
For the purpose of our generalization, it will be useful to think of this algebra as resulting from a sequence of objects $\mathcal{O}, L, L^{\otimes 2}, \ldots$ in the B-model category of holomorphic sheaves on the complex  manifold, obtained from the structure sheaf $\mathcal{O}$ by the successive application of the endofunctor $\Phi:\mathcal{E}\mapsto L\otimes\mathcal{E}$:
\begin{equation}\label{homalg}
A = \bigoplus_{n\geq 0} \Hom{\mathcal{O},\Phi^n(\mathcal{O})},
\end{equation}
with product $a\cdot b = \Phi^{|b|}(a)\circ b$.

We now observe that the above complex structure, B-brane $\mathcal{\OO}$, and endofunctor $\Phi$ all have nontrivial deformations in \emph{generalized complex geometry}~\cite{MR2013140,MR2811595}, as follows.  If $\sigma$ is a holomorphic Poisson tensor, we may construct a holomorphic family of generalized complex structures (Courant integrable complex structures on $TM\oplus T^*M$) parametrized by $\hbar\in\CC$,
\begin{equation}\label{defj}
\JJ_\hbar = \begin{pmatrix}-I&Q_\hbar \\0&I^*\end{pmatrix},
\end{equation}
where $Q_\hbar =-4\Im(\hbar\sigma)$. This generalized complex structure limits to the original complex structure as $\hbar\to 0$. The trivial line bundle continues to define a brane for all $\JJ_\hbar$, the \emph{trivial space-filling brane}.  Finally, if the holomorphic line bundle $L$ is a \emph{Poisson module}, meaning that $\sigma$ is lifted to a $\CC^*$-invariant Poisson structure on $L$, then, as shown in~\cite{gualtieri2010branes}, the operation of tensoring by $(L,\nabla)$ canonically deforms to a functor of the form
\begin{equation*}
\Phi_\hbar(\mathcal{B}) = (L,\nabla_\hbar)\otimes \varphi_{\hbar}^*(\mathcal{B}),
\end{equation*}
which is the composition of a diffeomorphism $\varphi_\hbar$ and a B-field gauge transformation (i.e. tensoring by the Hermitian line bundle $L$ equipped with a unitary connection $\nabla_\hbar$)
acting on the category of generalized complex branes for $\JJ_\hbar$. 

In principle, the deformation of $(I, \mathcal{O}, \Phi)$ into $(\JJ_{\hbar}, \mathcal{O}, \Phi_\hbar)$ described above should suffice for the construction of a deformed graded algebra, by using precisely the same technique as Equation~\ref{homalg}.  This would produce a version of the \emph{twisted homogeneous coordinate ring} often seen in noncommutative algebraic geometry~\cite{ARTIN1990249}.  The problem which stops us from reaching this final conclusion is that while the objects of the ``category'' of $\JJ_\hbar$-branes are well-defined geometrically~\cite{MR2811595}, the morphism spaces have not yet been defined in sufficient generality.  Thus we need to develop a definition of the morphism spaces $\Hom{\OO,\Phi^n_\hbar(\OO)}$ and of their compositions.

To define the vector space of morphisms between generalized complex branes, we make use of a fundamental tool in the study of Poisson geometry: the~\emph{Weinstein symplectic groupoid} $\mathcal{G}$ of the real Poisson structure $Q$ underlying the generalized complex structure $\JJ$~\cite{weinstein1987symplectic,MR1081010,MR1081011}.  
This groupoid, a symplectic manifold of twice the dimension of $M$, has motivated several recent advances in generalized complex and K\"ahler geometry~\cite{crainic2011,bailey2016integration,bischoff2018morita}, and we make use of it for the key purpose of converting a pair of generalized complex branes $(\BB_1,\BB_0)$ into a coisotropic A-brane $\BB_{1,0}$ in the symplectic groupoid $\mathcal{G}$.  The groupoid is, roughly speaking, the space of $Q$-Hamiltonian paths in $M$ modulo an appropriate notion of homotopy~\cite{cattaneo2001poisson, MR2128714}, and our A-brane $\BB_{1,0}$ is a reduction of the space of such paths going from $\BB_0$ to $\BB_1$.
The groupoid is also equipped with a canonical Lagrangian submanifold $\Ll$, the space of identity arrows representing constant paths at the points of $M$, and so we propose the following natural definition for the morphism space:
\begin{equation}\label{homsp}
\Hom[\JJ]{\BB_0,\BB_1} \coloneqq\Hom[\mathcal{G}]{\Ll,\BB_{1,0}}.
\end{equation}  
In short, we reduce the definition of morphisms between generalized complex branes in $M$ to that of morphisms between Lagrangian and space-filling coisotropic A-branes in an auxiliary real symplectic manifold $\mathcal{G}$.

The above proposal fits perfectly into the quantization paradigm developed by Gukov and Witten~\cite{Gukov:2008ve}. In their paper, a symplectic manifold is quantized by embedding it as a real Lagrangian submanifold of a symplectic complexification; for us, it is embedded as the identity bisection of the Weinstein symplectic groupoid. 

At present, coisotropic A-branes are not nearly as well-understood as Lagrangian submanifolds from a categorical point of view. In this paper, as for the other papers in which homomorphisms involving coisotropic branes are computed~\cite{Aldi:2005hz,Kapustin:2006pk, Gukov:2008ve}, we provide a method which may not be generally applicable, but which suffices for the problem at hand. Our proposed method for computing the Hom space~\eqref{homsp} involves choosing a \emph{multiplicative holomorphic Lagrangian polarization} of the space-filling brane $\BB_{1,0}$ and finding holomorphic Bohr-Sommerfeld leaves which intersect the Lagrangian brane. Here, we take inspiration from the substantial literature on $C^*$-algebra quantizations of real Poisson manifolds via symplectic groupoids, especially the papers \cite{MR1156554,MR2417440,MR3238532}.    

We implement this sequence of ideas for a general compact toric variety equipped with an invariant holomorphic Poisson structure induced by an R-matrix, and verify that the resulting algebra is a noncommutative deformation of the homogeneous coordinate ring of the underlying projective variety.  We take advantage of three special features which simplify the analysis: First, the symplectic groupoids of these Poisson structures can be explicitly constructed using the method of Xu~\cite{xu1993poisson}; second, the toric symmetry gives rise to a completely integrable system on the groupoid compatible with the multiplication, which provides a convenient holomorphic Lagrangian polarization of the space-filling branes in question.  Finally, the groupoid is, in this case, symplectomorphic to the real cotangent bundle of the toric variety, and hence is an exact symplectic manifold.  We are optimistic that none of these advantages are necessary for the future application of these ideas to more general situations.

\vspace{1ex}
\noindent{\bf Acknowledgements:}  We thank Denis Auroux, Joshua Lackman, James Pascaleff, David Mart\'{i}nez Torres, Michel Van den Bergh, Alan Weinstein, and Edward Witten for helpful discussions about this project over several years. F.B. is supported by an NSERC Postdoctoral Fellowship and M.G. is supported by an NSERC Discovery grant.



\section{A-Brane quantization of K\"ahler manifolds}
\subsection{K\"ahler metrics as LS submanifolds in twisted cotangent bundles} \label{metricbrane}

In this section, we review Donaldson's reformulation \cite{donaldson2002holomorphic} of a K\"{a}hler metric as an LS submanifold inside of a holomorphic symplectic affine bundle. Later, we will explain how this picture can be extended to the setting where the complex structure has been deformed into a generalized complex structure, following~\cite{bischoff2018morita}.

Let $(M, I, g)$ be a K\"{a}hler manifold, and let $\omega = gI$ be the associated K\"{a}hler form. The first step is to convert the data of the complex structure $I$ and the K\"{a}hler cohomology class $[\omega] \in H^{1,1}(M, \mathbb{R})$ into a holomorphic symplectic manifold $(Z, \Omega)$ deforming the holomorphic cotangent bundle. Here we consider the cotangent bundle $T^*M$ to be equipped with the holomorphic symplectic form $\Omega_{0}$ whose imaginary part is the canonical symplectic form on the real cotangent bundle. 

To begin, we build $(Z, \Omega)$ via a \v{C}ech construction. The K\"{a}hler form $\omega$ is a closed $2$-form of type $(1,1)$, and as a result, it is locally given by a K\"{a}hler potential $K$ as $\omega =  i \partial \overline{\partial} K$. Choose an open cover $M = \cup U_{i}$ such that over each open set there is a K\"{a}hler potential $K_{i} \in C^{\infty}(U_{i}, \mathbb{R})$. Then on the double overlap regions $U_{i} \cap U_{j}$, the $1$-forms $\mu_{ij} = \partial (K_{i} - K_{j})$ are closed and holomorphic. We now construct $Z$ by gluing $T^{*}U_{i}$ to $T^{*}U_{j}$ via the additive transition function $\mu_{ij}$. The resulting space $Z$ is no longer a vector bundle, but it retains an additive action of $T^*M$. In other words, $Z$ is an affine bundle for $T^*M$. Furthermore, $Z$ inherits a holomorphic symplectic structure $\Omega$ since the forms $\mu_{ij}$ are closed. There is a compatibility between this symplectic structure and the additive action of the cotangent bundle. To state this, note that the additive action defines a map 
\[
A : T^*M \times_{M} Z \to Z,
\]
and we may consider the graph of this map inside $(T^*M, -\Omega_{0}) \times (Z, -\Omega) \times (Z, \Omega)$. The compatibility condition states that the graph is a holomorphic Lagrangian submanifold. We summarize these properties by saying that $(Z, \Omega)$ is a holomorphic symplectic principal $T^*M$ bundle. 

Note that $(Z, \Omega)$ is classified by the \v{C}ech cohomology class $[\mu_{ij}] \in H^{1}(\Omega_{M}^{1, cl})$, where $\Omega^{1, cl}_{M}$ is the sheaf of closed holomorphic $1$-forms on $M$, which define local symmetries of the cotangent bundle. In fact, $\mu_{ij}$ is simply a \v{C}ech representative of the K\"{a}hler class $[\omega]$. From this perspective, we see that the symplectic principal bundle encodes the underlying holomorphic moduli of the K\"{a}hler structure, and deformations of $(Z, \Omega)$ correspond to deformations of both the complex structure and the K\"{a}hler class. 

The second step is to convert the data of the K\"{a}hler metric $g$ into a submanifold $\mathcal{L}$ of the space $(Z, \Omega)$. Again, we give a \v{C}ech description, making use of the local K\"{a}hler potentials $K_{i}$ that exist over the open sets of the cover. Over each open set $U_{i}$, define the local section $\mathcal{L}_{i} = - \partial K_{i}$ of $T^{*}U_{i}$. These sections satisfy the identity $\mathcal{L}_{i} + \mu_{ij} = \mathcal{L}_{j}$ on $U_{i} \cap U_{j}$, and hence they glue together to define a global section $\mathcal{L}$ of the projection map $\pi : Z \to M$. Throughout, we will abuse notation by identifying $\mathcal{L}$ with its image in $Z$. The section $\mathcal{L}$ gives a global meaning for the K\"{a}hler potentials. If we now pullback the holomorphic symplectic form $\Omega$ by the section, then we get 
\[
\mathcal{L}^{*} \Omega = \mathcal{L}_{i}^{*} \Omega_{0} =  i \partial \overline{\partial} K_{i} = \omega,
\]
where we use the local identification of $(Z, \Omega, \mathcal{L}) \cong (T^{*}U_{i}, \Omega_{0}, \mathcal{L}_{i})$ in the second equality, and the third equality follows from the defining property of the canonical symplectic form. This calculation shows that $\mathcal{L}$ is an LS submanifold of $Z$, meaning that it is Lagrangian with respect to the imaginary part of $\Omega$, and symplectic with respect to the real part of $\Omega$. It also shows that $\mathcal{L}$ allows us to recover the K\"{a}hler form $\omega$, and hence the metric $g$, from the principal bundle $(Z, \Omega)$. We therefore conclude that $\mathcal{L}$ encodes the K\"{a}hler metric. Under this equivalence, deforming $\mathcal{L}$ as an LS submanifold of $Z$ corresponds to varying $\omega$ within a fixed K\"{a}hler class. The degrees of freedom involved in describing a K\"{a}hler metric against the fixed background of a complex structure and K\"{a}hler class consists of the moduli of a single real-valued function on $M$. From the perspective of viewing the metric as an LS submanifold, this is due to the fact that Lagrangians deform via real generating functions. 

It is possible to give an alternative description of $(Z, \Omega, \mathcal{L})$. Namely, by using the additive action of the cotangent bundle on $\mathcal{L}$, we obtain the following isomorphism 
\[
\phi: T^{*}M \to Z, \qquad \alpha \mapsto \alpha + \mathcal{L}(\pi(\alpha)).
\]
This map sends the zero section to $\mathcal{L}$ and satisfies $\phi^{*}(\Omega) = \Omega_{0} + \pi^{*}\omega$. The map $\phi$ is not holomorphic. As a result the form $\phi^{*}(\Omega)$ is not holomorphic with respect to the given complex structure on the cotangent bundle. However, $\phi^{*}(\Omega)$ defines a new complex structure for which it is a holomorphic $(2,0)$-form. We may write the form as follows: 
\[
\Omega_{0} + \pi^{*}\omega = (\Re(\Omega_{0}) + \pi^{*}\omega) + i \Im(\Omega_{0}) = \Im(\Omega_{0}) I_{1} + i \Im(\Omega_{0}), 
\] 
where $I_{1} = \Im(\Omega_{0})^{-1} (\Re(\Omega_{0}) + \pi^{*}\omega)$ is a new complex structure on $T^{*}M$. Hence, we see that the symplectic principal bundle is simply given by a magnetic deformation of the cotangent bundle:
\begin{equation*}
(Z, \Omega) \cong (T^{*}M, \Omega_{0} + \pi^{*}\omega),
\end{equation*}
with the LS submanifold $\mathcal{L}$ given by the zero section $0_M$, and where the complex structure has been deformed in the manner described above. We can rescale the K\"{a}hler form by a positive real constant to get a family $(Z_{t}, \Omega_{t}, \mathcal{L}_{t})$ encoding the K\"{a}hler structures $t \omega$, for $t \in \mathbb{R}_{> 0}$. These are also given by magnetic deformations of the cotangent bundle, as follows:
\begin{equation} \label{magdef}
(Z_{t}, \Omega_{t}, \mathcal{L}_{t}) \cong (T^{*}M, \Omega_{0} + t \pi^{*}\omega, 0_M).
\end{equation}

The symplectic principal bundle $(Z, \Omega)$ described above is an example of a holomorphic symplectic Morita equivalence, a notion which is due to Weinstein and Xu~\cite{xu1991morita,MR2166451}. Morita equivalences arise in the study of Poisson manifolds and should, upon quantization, correspond to invertible bimodules over the quantized algebra of functions. In the present section, we have been dealing with the zero Poisson structure on $M$, and the Morita equivalences may be described as symplectic principal bundles. In Section~\ref{ncsection}, we will describe the generalization of Donaldson's reformulation of K\"{a}hler geometry to the case of non-trivial holomorphic Poisson manifolds--the theory of Morita equivalence then becomes essential. 

\subsection{K\"ahler forms from A-Brane pairs} \label{Braneprequant}

In this paper, we take the perspective of viewing the twisted cotangent bundle $Z$ and its LS submanifold $\mathcal{L}$ as a pair of branes in the extended A-model of the real cotangent bundle of $M$. This point of view is an instance of the general philosophy on geometric quantization developed by Gukov and Witten \cite{Gukov:2008ve}.  The Lagrangian submanifold $\mathcal{L}$ is called a \emph{brane} when it is equipped with a flat unitary bundle --- in our case we simply take it to be the trivial bundle.  More interesting is the interpretation of $Z$ as a space-filling brane in $T^*M$.  
 Following the work of Kapustin-Orlov on Coisotropic A-branes~\cite{Kapustin:2001ij}, the holomorphic symplectic structure on $Z$ may be viewed as a space-filling brane in the following sense.

\begin{definition}
A (rank 1) \emph{space-filling A-brane} on the real symplectic manifold $(X,\varpi)$ is a Hermitian line bundle $U$ equipped with a unitary connection $\nabla$ whose curvature $-2\pi i F$ satisfies the condition $(\varpi^{-1} F)^2 = -1$, or equivalently that the complex 2-form $\Omega_F = F + i\varpi$ defines a holomorphic symplectic structure on $X$.\footnote{Setting $\dim_{\RR} X = 4k$, this condition may be expressed as the vanishing of $\Omega_F^{k+1}$ together with the nonvanishing of $\Omega_F^k\wedge\ol{\Omega}_F^k$.}
\end{definition}

Indeed, suppose the cohomology class of the K\"ahler form $\omega$ is integral, so that it is possible to find a prequantization $L$, i.e. a holomorphic Hermitian line bundle such that the curvature of the associated Chern connection $\nabla$ is $-2 \pi i \omega$. Since the real part of the holomorphic symplectic structure on $Z$ is $\Re(\Omega) = \Re(\Omega_0) + \pi^*\omega$, and since $\Omega_0$ has canonical primitive $\alpha_0$, we obtain a unitary bundle 
\begin{equation}\label{sfbct}
(U,\hat\nabla) = (\pi^*L, \pi^*\nabla - 2\pi i\Re(\alpha_0))
\end{equation}
making $Z$ into a space-filling A-brane on the real cotangent bundle. We refer to this brane as the \emph{canonical coisotropic}, denoted $\mathcal{B}^{cc}$. 
Having defined the pair of branes $\mathcal{L}, \mathcal{B}^{cc}$, we may recover the original prequantization of the K\"{a}hler structure by taking their intersection and tensor quotient, as follows:
\[
(L, \nabla) \cong (U, \hat\nabla)|_{\mathcal{L}} \otimes (\mathcal{O}, d)^{*}. 
\]

\subsection{Holomorphic prequantization of space-filling A-branes}\label{holpq}

The exactness of the symplectic form on the cotangent bundle gives rise to a useful interpretation of our space-filling branes as holomorphic prequantizations.\footnote{The physical relevance of this interpretation is that the action of the A-model reduces to the holonomy, along the 1-dimensional boundary, of the corresponding non-unitary connection.}

\begin{definition}
Let $(U,\nabla)$ be a space-filling A-brane on the symplectic manifold $(X,\varpi)$, with curvature $-2\pi iF$.  When $\varpi$ is exact, i.e. $\varpi = dA$, we may define the complex (non-unitary) connection $D= \nabla + 2\pi A$, whose curvature is the $(2,0)$-form $-2\pi i(F + i\varpi) = -2\pi i\Omega_F$.  This connection endows $U$ with the structure of a holomorphic line bundle with holomorphic connection $D$.  We refer to this holomorphic pair $(U,D)$ as the \emph{holomorphic prequantization} of the brane. 
\end{definition}
We may apply this to our situation, since the cotangent bundle $(T^{*}M, \Im(\Omega_{0}))$ is an exact symplectic manifold, with canonical primitive $\Im(\alpha_{0})$. 
The resulting holomorphic prequantization of the space filling brane~\eqref{sfbct} is then
\[
(U, D) =  (U,\hat\nabla + 2\pi\Im(\alpha_0)) =(\pi^{*} L, \pi^{*} \nabla - 2 \pi i \alpha_{0}). 
\]
The curvature of this connection is $- 2 \pi i \Omega$, a closed holomorphic $(2,0)$-form. Hence, $D$ defines a holomorphic structure on $U$, and a holomorphic prequantization of $\Omega$. We summarise the content of the previous two sections in the following proposition. 

\begin{proposition} \label{prequantizedbranes}
The data of a K\"{a}hler structure $(g, I, \omega)$ on a manifold $M$, together with a unitary prequantization $(L, \nabla)$ of $\omega$, determine a pair of branes in the real cotangent bundle $T^*M$:
\begin{enumerate}
\item a Lagrangian brane $\mathcal{L}$ given by the zero section equipped with the trivial flat connection, and 
\item a space-filling brane $\mathcal{B}^{cc}$ given by the unitary line bundle $(U,\hat\nabla) = (\pi^*L,\pi^*\nabla - 2\pi i\Re(\alpha_0))$, with corresponding holomorphic prequantization $(U, D) =  (\pi^{*} L, \pi^{*} \nabla - 2 \pi i \alpha_{0})$, where $\alpha_{0}$ is the tautological holomorphic $(1,0)$-form on the holomorphic cotangent bundle. 
\end{enumerate}
Conversely, the pair of branes uniquely determines the prequantized K\"{a}hler structure by taking intersection and tensor quotient. Under this correspondence, Hamiltonian deformations of $\mathcal{L}$ correspond to variations of $\omega$ within a fixed K\"{a}hler class, with the accompanying variation of the prequantization. 
\end{proposition}



\subsection{Prequantization via holomorphic symplectic reduction} \label{reductionsection}
In Section \ref{metricbrane} we gave two constructions of the holomorphic symplectic deformation of the cotangent bundle $(Z, \Omega)$ and its smooth $\Im(\Omega)$-Lagrangian, $\Re(\Omega)$-symplectic (LS) submanifold $\mathcal{L}$: a \v{C}ech construction and a magnetic deformation construction. In this section we present a third construction via symplectic reduction which is available when the K\"{a}hler form $\omega$ has been prequantized. This construction has the advantage of uniformly constructing the spaces $(Z_{t}, \Omega_{t})$ and $\mathcal{L}_{t}$ associated to the rescalings $t \omega$ for $t \in \mathbb{R}$, along with their associated holomorphically prequantized brane structures when $t$ is integral. The results of this section are summarized in Proposition~\ref{reductionprop}. 

Let $P_{L}$ be the principal $\mathbb{C}^{*}$-bundle obtained from $L$ by removing its zero section, and let $\mathcal{E}$ be the Euler vector field, which generates the action by $\mathbb{C}^{*}$. The cotangent bundle $T^{*}P_{L}$ is a Hamiltonian $\mathbb{C}^{*}$-space, with moment map given by pairing covectors with $i \mathcal{E}$. Holomorphic symplectic reduction at the value $0$ recovers the holomorphic cotangent bundle of $M$:
\[
(T^{*}M, \Omega_{0}) \cong \mathcal{E}^{-1}(0)/\mathbb{C}^{*}. 
\]
Reducing at any non-zero value produces a holomorphic symplectic principal $T^*M$ bundle:
\[
(Z_{t}, \Omega_{t}) := \mathcal{E}^{-1}(\tfrac{t}{2\pi})/\mathbb{C}^{*}. 
\]

When $t \in \mathbb{R}$, we can show that $(Z_{t}, \Omega_{t})$ is the symplectic principal bundle encoding the K\"{a}hler class of $t \omega$ by producing a suitable LS submanifold. This is because the choice of such a submanifold allows us to identify $Z_{t}$ with a magnetic deformation of the cotangent bundle, as in Equation \ref{magdef}. To produce the LS submanifold, let $h$ be the Hermitian metric on $L$ whose Chern connection has curvature $-2 \pi i \omega$. Let $S \subset P_{L}$ be the corresponding unit $S^{1}$-bundle, and let $C_{S}$ denote its conormal bundle in $T^{*}P_{L}$, a Lagrangian submanifold with respect to the canonical symplectic form $\Im(\Omega_{0})$ on $T^*P_{L}$. Since $S$ is a real codimension $1$ submanifold of $P_{L}$, defined by the equation $h = 1$, its conormal bundle is a real line bundle over $S$, with global generating section $\frac{1}{2}d\log(h)$. This produces an isomorphism 
\begin{equation} \label{conormaltrivialiso}
S \times \mathbb{R} \to C_{S}, \qquad (z, \tau) \mapsto \tfrac{\tau}{2}d\log(h)_{z},
\end{equation}
for which the function $\mathcal{E}$ is given by projection onto the factor $\mathbb{R}$. We can thereby see that $C_{S}$ is invariant under $S^{1} \subseteq \mathbb{C}^{*}$, which acts only on the first factor $S$. We can reduce $C_{S}$ to produce Lagrangian submanifolds in $(Z_{t}, \Im(\Omega_{t}))$, for $t \in \mathbb{R}$. Explicitly, these are given by
\[
\mathcal{L}_{t} = (C_{S} \cap \mathcal{E}^{-1}(\tfrac{t}{2\pi}))/S^{1} \cong S/S^{1} = M.
\]
The LS submanifold $\Ll_t\subset Z_t$ defines a section of the map $\pi:Z_t\to M$, which we also refer to as $\Ll_t$.

In order to determine the pullback $\mathcal{L}_{t}^*(\Omega_{t})$, we first pullback $\alpha_{0}$, the primitive of $\Omega_{0}$ on $T^{*}P_{L}$, to $C_{S}$. Using the above isomorphism \ref{conormaltrivialiso}, this gives 
\[
\alpha_{0}|_{C_{S}} = i\tau \partial \log(h)|_{S}. 
\]
Restricting this to constant $\tau$, we get the contact structure on the $S^1$-bundle $p: S \to M$ which prequantizes $\omega$. Hence, 
\[
\Omega_{0}|_{C_{S}}= i d\tau \wedge \partial \log(h)|_{S} + i\tau \overline{\partial} \partial \log(h)|_{S} =  i d\tau \wedge \partial \log(h)|_{S} + 2\pi \tau p^{*} \omega,
\]
is the `symplectization' of the contact structure. The $S^{1}$-action is now Hamiltonian with respect to this symplectic form (for $\tau \neq 0$), with moment map given by projection to the factor $\mathbb{R}$ (which is equivalent to pairing covectors with $\mathcal{E}$). Hence $\Ll_t^*(\Omega_{t})$ is obtained by symplectic reduction, and is equal to 
\[
\mathcal{L}_{t}^*(\Omega_{t}) = t \omega. 
\]
This shows that $(Z_{t}, \Omega_{t})$ and $\mathcal{L}_{t}$ coincide with the principal bundles with sections constructed in Section~\ref{metricbrane} via magnetic deformation as in Equation \ref{magdef}. 

Rather than describing unitary branes by reduction, it is convenient to directly build their holomorphic prequantizations: to construct the prequantizations of $(Z_{t}, \Omega_{t})$ we start with a prequantization of $(T^{*}P_{L}, \Omega_{0})$ and attempt to reduce it to the various $Z_{t}$. The prequantization of the cotangent bundle is given by $(\mathcal{O}, d - 2 \pi i \alpha_{0})$. In order to reduce it, we restrict it to a level set $\mathcal{E}^{-1}(\frac{t}{2\pi})$ and attempt to descend it to the $\mathbb{C}^{*}$-quotient. This means that we must find a holomorphic line bundle with connection $(U_{t}, D_{t})$ on $Z_{t}$, and a flat isomorphism $(\mathcal{O}, d - 2\pi i \alpha_{0})|_{\mathcal{E}^{-1}(\frac{t}{2\pi})} \cong q^{*}(U_{t}, D_{t})$, where $q: \mathcal{E}^{-1}(\frac{t}{2\pi}) \to Z_{t}$ is the quotient map. In other words, the sheaf of sections of $(U_{t}, D_{t})$ can be defined to be the sections of $(\mathcal{O}, d - 2\pi i \alpha_{0})|_{\mathcal{E}^{-1}(\frac{t}{2\pi})}$ that are flat along the $\mathbb{C}^{*}$-orbits. A calculation shows that the monodromy of $d - 2\pi i \alpha_{0}$ along the $\mathbb{C}^{*}$-orbits is $\exp(-2 \pi i t)$. Hence, the prequantization descends to $Z_{t}$ if and only if $t \in \mathbb{Z}$. 

We summarise the constructions of this section in the following proposition, which also serves as a definition. 
\begin{proposition} \label{reductionprop}
Let the K\"{a}hler manifold $(M, I, \omega)$ be prequantized by the holomorphic line bundle $L$, equipped with Hermitian metric $h$ and unitary connection $\nabla$. Let $P_{L}$ denote the associated principal $\mathbb{C}^{*}$-bundle, and let $(T^{*}P_{L}, \Omega_{0})$ be its cotangent bundle, which is equipped with the holomorphic prequantization $(\mathcal{O}, d - 2 \pi i \alpha_{0})$. Furthermore, let $\mathcal{E}$ be the vector field on $P_{L}$ which generates the $\mathbb{C}^{*}$-scaling action, viewed as a function on $T^*P_{L}$. 
\begin{enumerate}
\item The symplectic reduced spaces, for $t\in\RR$,                                                                                          
\[
(Z_{t}, \Omega_{t}) = \mathcal{E}^{-1}(\tfrac{t}{2\pi})/\mathbb{C}^{*}, 
\]
are holomorphic symplectic principal $T^*M$ bundles which encode the K\"{a}hler class $t[ \omega]$. 
\item For $n \in \mathbb{Z}$, the spaces $(Z_{n}, \Omega_{n})$ are equipped with a holomorphic prequantization, defined by reducing the prequantization of $T^{*}P_{L}$
\[
(U_{n}, D_{n}) = (\mathcal{O}, d - 2 \pi i \alpha_{0})|_{ \mathcal{E}^{-1}(\frac{n}{2\pi}) }/\mathbb{C}^{*}.
\]
\item There are canonical holomorphic isomorphisms of line bundles 
\[
\phi_n:U_{n} \stackrel{\cong}{\longrightarrow} \pi^{*}(L^{\otimes n}), 
\]
where $\pi: Z_{n} \to M$ is the principal bundle projection. 
\end{enumerate}
Let $S \subset P_{L}$ be the unit $S^{1}$-bundle defined by the metric $h$, and let $C_{S} \subset T^{*}P_{L}$ denote its conormal bundle, which is Lagrangian with respect to $\Im(\Omega_{0})$. 
\begin{enumerate}
\item The reduction of $C_{S}$ in $Z_{t}$
\[
\mathcal{L}_{t} = (C_{S} \cap \mathcal{E}^{-1}(\tfrac{t}{2\pi}))/S^{1}
\]
is an LS submanifold which encodes the K\"{a}hler form $t \omega$. 
\item For $n \in \mathbb{Z}$, the morphism $\phi_n$ induces an isomorphism of complex line bundles with connection
\[
\Ll_n^*(\phi_n):\mathcal{L}_{n}^*(U_{n}, D_{n}) \stackrel{\cong}{\longrightarrow} (L^{\otimes n}, \nabla^{\otimes n}).  
\]
\end{enumerate}
\end{proposition}
\begin{proof}
What remain to be established are the canonical isomorphisms between $U_{n}$ and $\pi^{*}(L^{\otimes n})$. These arise because there is a tautological section $s_{n}$ of $p^{*}(L^{\otimes n})$, where $p: P_{L} \to M$, which sends an element $\lambda_{m} \in P_{L}$ to the element $(\lambda_{m})^{\otimes n} \in L^{\otimes n}|_{m}$. If we further pull this back to $T^{*}P_{L}$, and then restrict to $\mathcal{E}^{-1}(\frac{n}{2\pi})$, we obtain a map $\phi'_{n}:  \mathcal{O}_{\mathcal{E}^{-1}(\frac{n}{2\pi})} \to \pi^{*}L^{\otimes n}$ covering $q: \mathcal{E}^{-1}(\frac{n}{2\pi}) \to Z_{n}$. Hence we need only show that this map is invariant under the $\mathbb{C}^{*}$-action on the domain. But this follows because $\nabla = d - 2 \pi i \alpha_{0}$ generates the weight $-n$ $\mathbb{C}^{*}$-action on $ \mathcal{O}|_{\mathcal{E}^{-1}(\frac{n}{2\pi})}$. Therefore, by reducing $\phi_n'$, we obtain the required isomorphism $\phi_n$ between $U_{n}$ and $\pi^{*}(L^{\otimes n})$.  

Pulling back to the LS submanifold we obtain an isomorphism $\mathcal{L}_{n}^{*}(\phi_{n})$ between $\mathcal{L}_{n}^*U_{n}$ and $L^{\otimes n}$. That this identifies the connections $D_{n}|_{\mathcal{L}_{n}}$ and $\nabla^{\otimes n}$ follows immediately from the fact that the pullback of $\alpha_{0}$ to $C_{S} \cap \mathcal{E}^{-1}(\frac{n}{2\pi}) \cong S$ is the canonical contact form prequantizing $n \omega$. 
\end{proof}

\subsection{Torus equivariance} \label{toricsection}
In this section we focus on the case of toric K\"{a}hler manifolds, and show how the toric structures lift to equivariant structures on the prequantized branes.

Let $(M, I, \omega)$ be a compact connected toric K\"{a}hler manifold of complex dimension $n$. This means that $M$ is equipped with a holomorphic effective action of a complex torus $\mathbb{T}_{\mathbb{C}} \cong (\mathbb{C}^{*})^{n}$, such that the action of the real compact subgroup $\mathbb{T} \cong (S^{1})^{n}$ is Hamiltonian, with moment map $\mu : M \to \mathfrak{t}^{*}$, where $\mathfrak{t}$ is the Lie algebra of $\mathbb{T}$. Therefore, the infinitesimal action $X : \mathfrak{t} \to \mathfrak{X}(M)$ is determined by the moment map via the equation 
\[
X_{a} = - \omega^{-1}(d \langle \mu, a \rangle ),
\]
for $a \in \mathfrak{t}$. The real Lie algebra $\mathfrak{t}$ is embedded into its complexification $\mathfrak{t}_{\mathbb{C}}$ as $i \mathfrak{t} \subset \mathfrak{t} \otimes \mathbb{C}$. The infinitesimal holomorphic action is given by 
\[
V_{a + ib} = \tfrac{1}{2}( (X_{b} - IX_{a}) - i(IX_{b} + X_{a})). 
\]
The kernel of the exponential map defines a lattice $2 \pi i \Lambda \subset \mathfrak{t}_{\mathbb{C}}$. We view $\Lambda \subset \mathfrak{t} \subset \mathfrak{t}_{\mathbb{C}}$, so that $2 \pi i \Lambda \subset i\mathfrak{t}$. A basis $e_{1}, ..., e_{n} \in \Lambda$ determines an isomorphism $\mathbb{T}_{\mathbb{C}} \cong (\mathbb{C}^{*})^{n}$ and a basis of holomorphic vector fields $V_{i} = V_{e_{i}}$. For simplicity let us fix such a basis. The image of the moment map $\Delta = \mu(M)$ is called the moment polytope, and it is given by the convex hull of $\mu(p)$, for $p$ the fixed points of the action. 

Given a fixed point $p$ there is a $\mathbb{T}_{\mathbb{C}}$-invariant chart $U_{p} \cong \mathbb{C}^{n}$ centred at $p$ which linearizes the torus action. This means that there are coordinates $z$ on $U_{p}$ such that the torus action takes the form 
\[
(w_{1}, ..., w_{n}) \ast (z_{1}, ..., z_{n}) = (\Pi_{i} w_{i}^{\lambda_{i}^{(1)}} z_{1}, ..., \Pi_{i} w_{i}^{\lambda_{i}^{(n)}} z_{n}),
\]
and the vector field $V_{i}$ is given by 
\[
V_{i} = \sum_{j} \lambda_{i}^{(j)} z_{j} \partial_{z_{j}},
\]
where $\lambda^{(1)}, ..., \lambda^{(n)}$ form a basis of the dual lattice $\Lambda^{*} \subset \mathfrak{t}^{*}$. By using an averaging argument, we can show that there is a $\mathbb{T}$-invariant K\"{a}hler potential $K_{p}$ for $\omega$ on $U_{p}$, which is unique if we require $K_{p}(p) = 0$. The moment map $\mu$ may be expressed in terms of the derivatives of this K\"{a}hler potential as follows
\begin{equation} \label{mommappot}
\mu - \mu(p) = \sum_{i} V_{i}(K_{p}) e^{i} = \sum_{j} z_{j} \tfrac{\partial K_{p}}{\partial z_{j}} \lambda^{(j)},
\end{equation}
where $e^{i}$ form the basis of $\Lambda^{*}$ dual to the given basis $e_{i}$. 

The holomorphic action of $\mathbb{T}_{\mathbb{C}}$ on $M$ lifts naturally to a holomorphic Hamiltonian action on $(T^{*}M, \Omega_{0})$ with moment map 
\[
J_{0} : T^{*}M \to \mathfrak{t}^{*}_{\mathbb{C}},
\]
defined for $\alpha \in T^{*}M$ and $u \in \mathfrak{t}_{\mathbb{C}}$ by the equation $ \langle J_{0}(\alpha), u \rangle = i \langle \alpha, V_{u} \rangle$. The torus action can also be lifted to a holomorphic Hamiltonian action on $Z_{t}$. Using the presentation of $Z_{t}$ as a magnetic deformation of the cotangent bundle (Equation \ref{magdef}), the moment map is given by
\[
J_{t} = J_{0} - i t \pi^{*}\mu : Z_{t} \to \mathfrak{t}_{\mathbb{C}}^{*}. 
\]
We denote by $V_{u}^{t}$ the Hamiltonian vector field on $Z_{t}$ associated to the function $\langle J_{t}, u \rangle$, for $u \in \mathfrak{t}_{\mathbb{C}}$. As a result of the following proposition, it projects to $V_{u}$ on $M$. 

\begin{proposition} \label{Hamiltonianlift}
The map $J_{t}$ is a holomorphic moment map for an effective holomorphic action of $\mathbb{T}_{\mathbb{C}}$ which lifts the torus action on $M$. This makes $(Z_{t}, \Omega_{t})$ into a holomorphic completely integrable system. Furthermore, the brane $\mathcal{L}_{t}$ is invariant under the compact torus $\mathbb{T}$, and $J_{t}$ restricts to the real moment map for $t \omega = \mathcal{L}_{t}^*(\Omega_{t})$: 
\[
J_{t}|_{\mathcal{L}_{t}} = -i t \mu. 
\]
\end{proposition}

\begin{proof}
In this proof, we use the magnetic deformation construction of $(Z_{t}, \Omega_{t}, \mathcal{L}_{t})$ given in Equation \ref{magdef}. The torus invariant charts $U_{p}$ cover $M$ and their pre-images $Z_{t}|_{U_{p}} = \pi^{-1}(U_{p})$ cover $Z_{t}$. Hence, we prove the theorem by checking the statements locally in each of these open sets. Over the open set $U_{p}$ we have the canonical torus-invariant K\"{a}hler potential $K_{p}$ and this determines a holomorphic symplectomorphism
\[
\varphi^{t}_{p} : (T^{*}U_{p}, \Omega_{0}) \to (Z_{t}|_{U_{p}}, \Omega_{t}), \qquad \alpha_{x} \mapsto \alpha_{x} + t \partial K_{p}. 
\]
Pulling back $J_{t}$ we get 
\begin{align*}
J_{t} \circ \varphi^{t}_{p}(\alpha_{x}) &= J_{0}(\alpha_{x} + t \partial K_{p}) - i t \mu(x) \\
&= J_{0}(\alpha_{x}) + it \sum_{i} V_{i}(K_{p})e^{i} - it \mu(x) \\
&= J_{0}(\alpha_{x}) - i t \mu(p),
\end{align*}
where, in the third line, we are using the expression in Equation \ref{mommappot} for the moment map given in terms of the derivatives of the K\"{a}hler potential. The claims about the holomorphic action on $(Z_{t}, \Omega_{t})$ then follow from the corresponding facts about the action on $(T^{*}M, \Omega_{0})$. 

The brane bisection $\mathcal{L}_{t}$ pulls back under $\varphi^{t}_{p}$ to the graph of $- t \partial K_{p}$. This is invariant under the action of the compact torus since both $K_{p}$ and the complex structure are. 

Finally, to establish the claim about the restriction of the moment map to $\mathcal{L}_{t}$, note that in the magnetic deformation construction, this submanifold is simply given by the zero section in the cotangent bundle. Hence, $J_{0}|_{\mathcal{L}_{t}} = 0$, and the claim follows. 
\end{proof}


\begin{remark}
The holomorphic completely integrable system $(Z_{1}, \Omega_{1}, J_{1})$ defines a canonical complexification of the toric K\"{a}hler manifold $(M, I, \omega, \mu)$. The original toric structure can then be viewed, via an equivariant generalization of Gukov-Witten \cite{Gukov:2008ve}, as the intersection of two equivariant branes. The action of the compact torus $\mathbb{T}$ turns out not to be essential, and Proposition \ref{Hamiltonianlift} may be generalized to the setting of any completely integrable system on a K\"{a}hler manifold which is compatible with the complex structure. 
\end{remark}

We now assume that the cohomology class of $\omega$ is integral, or equivalently that the edges of the moment polytope $\Delta$ lie in $\frac{1}{2\pi} \Lambda^{*}$. Let $(L, h, \nabla)$ be the prequantization. Using Kostant's prescription, the torus action can be lifted to a holomorphic equivariant action on $L$. Infinitesimally, it is given by 
\[
A(a) = \nabla_{V_{a}} + 2 \pi \langle \mu, a \rangle, 
\]
for $a \in \mathfrak{t}_{\mathbb{C}}$. Note that only the action of $i \mathfrak{t}$ preserves the metric. This infinitesimal action integrates to an equivariant action of $\mathbb{T}_{\mathbb{C}}$ if the vertices of the moment polytope lie in the lattice $\frac{1}{2\pi} \Lambda^{*}$, and this can be achieved by shifting the moment map $\mu$ by a constant.

We can now repeat the constructions from Section \ref{reductionsection} in the presence of the $\mathbb{T}_{\mathbb{C}}$ action to produce the equivariant version of Proposition \ref{reductionprop}. Namely, the $\mathbb{T}_{\mathbb{C}}$ action on the associated $\mathbb{C}^{*}$-bundle $P_{L}$ lifts to a Hamiltonian action on $T^{*}P_{L}$ which commutes with the existing Hamiltonian $\mathbb{C}^{*}$-action. This Hamiltonian action therefore passes to the reduced spaces $(Z_{t}, \Omega_{t})$, recovering the results of Proposition \ref{Hamiltonianlift}. Furthermore, the torus action on $T^{*}P_{L}$ lifts to the trivial equivariant action on the prequantization $(\mathcal{O}, d - 2 \pi i \alpha_{0})$, and this also descends to the reduced spaces. We therefore obtain the following proposition. 

\begin{proposition} \label{liftedequivariantaction}
The $\mathbb{T}_{\mathbb{C}}$-action on $(Z_{n}, \Omega_{n})$ lifts to an equivariant action on the prequantization $(U_{n}, D_{n})$. Infinitesimally, this is given by the holomorphic version of Kostant's prescription:
\[
A_{n}(a) = D_{n}(V_{a}^{n}) + 2 \pi i \langle J_{n}, a \rangle,
\]
for $a \in \mathfrak{t}_{\mathbb{C}}$. Furthermore, the canonical isomorphism $U_{n} \cong \pi^{*}(L^{\otimes n})$ from Proposition \ref{reductionprop} is equivariant. 
\end{proposition}

\subsection{K\"ahler quantization} \label{thequantization}
In Section \ref{Braneprequant} we interpreted a prequantized K\"{a}hler structure as a pair of branes $\mathcal{B}^{cc}$ and $\mathcal{L}$ in the A-model of the cotangent bundle of $M$. Following Gukov-Witten \cite{Gukov:2008ve}, we now interpret the vector space of global holomorphic sections of $L$ in terms of a space of homomorphisms between the branes in the extended A-model:
\[
H^{0}(M, L) \cong \Hom{\mathcal{L}, \mathcal{B}^{cc}}. 
\]
In order to realise this isomorphism, we first choose a ``real" polarization of the holomorphic symplectic variety $(Z, \Omega)$ underlying $\mathcal{B}^{cc}$. By this we mean a holomorphic Lagrangian foliation $\mathcal{F}$ of $(Z, \Omega)$. We then consider the subset of leaves of this foliation which intersect the brane $\mathcal{L}$. Our proposal for the space of homomorphisms is the space of holomorphic sections of $(U,D)$ restricted to this subset which are covariantly constant along the foliation. It will be useful in this section to work with the sequence of prequantized branes $(Z_{n}, \Omega_{n}, \mathcal{L}_{n}, U_{n}, D_{n})$, for $n \in \mathbb{Z}$, which was developed in Section \ref{reductionsection}. In this case, the space of homomorphisms $\Hom{\mathcal{L}_{n}, \mathcal{B}^{cc}_{n}}$ will be shown to be isomorphic to the global holomorphic sections of $L^{\otimes n}$.

A natural choice for the polarization is given by the projection map $\pi : Z_{n} \to M$, whose fibres are Lagrangian affine spaces, and all of which intersect $\mathcal{L}_{n}$. A holomorphic section of $(U_{n},D_{n})$ which is constant in the fibrewise directions is then easily seen to arise as the pullback of a holomorphic section of $L^{\otimes n}$. Hence this realises the desired isomorphism.

When $(M, I, \omega)$ is a toric K\"{a}hler structure, there is an alternate polarization of $Z_{n}$ which is better adapted to the multiplicative structures that we will eventually consider. It is given by the fibres of the holomorphic moment map $J_{n} : Z_{n} \to \mathfrak{t}^{\ast}_{\mathbb{C}}$ constructed in Section \ref{toricsection}. The fibres of this map are connected but not simply connected, and therefore the monodromy of the connection $D_{n}$ restricted to these fibres may be non-trivial. This obstructs the existence of holomorphic sections which are covariantly constant along the polarization. A fibre $J_{n}^{-1}(\xi)$ is a \emph{Bohr-Sommerfeld} fibre if the monodromy of the connection $(U_{n}, D_{n})|_{J_{n}^{-1}(\xi)}$ is trivial. Let $\mathcal{F}_{BS}(\mathcal{L}_{n}, \mathcal{B}_{n}^{cc})$ denote the set of Bohr-Sommerfeld fibres which have a non-trivial intersection with $\mathcal{L}_{n}$. In this setting, the space of homomorphisms in the A-model is defined to be the direct sum over $\mathcal{F}_{BS}(\mathcal{L}_{n}, \mathcal{B}_{n}^{cc})$ of the vector spaces of covariantly constant sections of $U_{n}$ restricted to the corresponding Bohr-Sommerfeld fibres: 
\begin{equation} \label{proposalforquant}
\Hom{\mathcal{L}_{n}, \mathcal{B}^{cc}_{n}} \coloneqq \bigoplus_{J_{n}^{-1}(\xi) \in \mathcal{F}_{BS}(\mathcal{L}_{n}, \mathcal{B}_{n}^{cc})} H^{0}(J_{n}^{-1}(\xi), U_{n})^{D_{n}}. 
\end{equation}
Each summand consists of the flat sections of a line bundle, and hence is a $1$-dimensional vector space. In order to determine this space of homomorphisms, we therefore need to determine the set $\mathcal{F}_{BS}(\mathcal{L}_{n}, \mathcal{B}_{n}^{cc})$ of Bohr-Sommerfeld fibres which intersect the brane $\mathcal{L}_{n}$. 
\begin{proposition} \label{BSdetermination}
Consider the polarization of $(Z_{n}, \Omega_{n})$ defined by the holomorphic moment map $J_{n}: Z_{n} \to \mathfrak{t}_{\mathbb{C}}^{*}$. A fibre $J_{n}^{-1}(\xi)$ is Bohr-Sommerfeld for the prequantization $(U_{n}, D_{n})$ if and only if $\xi \in \frac{1}{2\pi i} \Lambda^{*}$. Furthermore, the fibre $J_{n}^{-1}(\xi)$ has non-trivial intersection with the brane $\mathcal{L}_{n}$ if and only if $\xi \in -i n \Delta$, where $\Delta  = \mu(M)$ is the image of the real moment map $\mu$. As a result, the space of homomorphisms between $\mathcal{B}^{cc}_{n}$ and $\mathcal{L}_{n}$ is given by 
\[
\Hom{\mathcal{L}_{n}, \mathcal{B}^{cc}_{n}} = \bigoplus_{ \lambda \in \Lambda^{*} \cap 2 \pi n \Delta } H^{0}(J_{n}^{-1}(\tfrac{\lambda}{2\pi i}), U_{n})^{D_{n}} \cong  \bigoplus_{ \lambda \in \Lambda^{*} \cap 2 \pi n \Delta } \mathbb{C} \lambda. 
\]
\end{proposition}
\begin{proof}
To begin, Proposition \ref{Hamiltonianlift} states that $J_{n}|_{\mathcal{L}_{n}} = -in \mu$, which has image $-in \Delta$. From this it follows that the fibre $J_{n}^{-1}(\xi)$ intersects $\mathcal{L}_{n}$ if and only if $\xi \in -in \Delta$. 

In order to determine the geometry of the fibres, we make use of the local isomorphisms $\varphi_{p}^{n}: (T^{*}U_{p}, \Omega_{0}) \to (Z_{n}|_{U_{p}}, \Omega_{n})$ constructed in the proof of Proposition \ref{Hamiltonianlift}. Recall that $(\varphi^{n}_{p})^{*}(J_{n}) = J_{0} - i n \mu(p)$. Using the linearized coordinates $(z_{i})$ on $U_{p}$ introduced in Section \ref{toricsection} we get coordinates $(z_{i}, p_{i})$ on $T^{*}U_{p}$ for which 
\[
(\varphi^{n}_{p})^{*}(J_{n}) = i \sum_{i} z_{i} p_{i} \lambda^{(i)} - in \mu(p). 
\]
From this description, we see that the fibres of $J_{n}|_{U_{p}}$ are submanifolds of the form 
\[
\{ (z, p) \ | \ zp = 0 \}^{k} \times (\mathbb{C}^{*})^{n-k},
\]
for some values of $k$ which depend on the specific fibre. If we further restrict $J_{n}$ to the locus where the coordinates $z_{i} \neq 0$, then the fibres are isomorphic to the complex torus $\mathbb{T}_{\mathbb{C}}$. The difference between the full fibres and the fibres over the restricted locus consists of contractible components. Hence, for the purpose of determining the Bohr-Sommerfeld fibres, it suffices to restrict to the locus $z_{i} \neq 0$. More invariantly, we may restrict $J_{n}$ to $Z_{n}|_{M^{o}}$, where $M^{o} \subset M$ is the dense open subset where the complex torus action is free and transitive. Note that we have a canonical isomorphism 
\[
\pi \times J_{n} : Z_{n}|_{M^{o}} \to M^{o} \times \mathfrak{t}^{*}_{\mathbb{C}},
\]
which shows that the fibres of $J_{n}|_{M^{o}}$ are isomorphic to $M^{o} \cong \mathbb{T}_{\mathbb{C}}$. These fibres are torus invariant, and the Hamiltonian vector fields $V_{a}^{n}$, for $a \in \mathfrak{t}_{\mathbb{C}}$, restrict to global generators of their tangent bundles. In order to determine the monodromy, recall from Proposition \ref{liftedequivariantaction} that the connection $D_{n}$ has the form 
\[
D_{n}(V_{a}^{n}) = A_{n}(a) - 2\pi i \langle J_{n}, a \rangle,
\]
where $A_{n}$ defines the equivariant structure. Since the $A_{n}$-term defines a torus action, it does not contribute to the monodromy. Hence, the monodromy for the fibre $J_{n}^{-1}(\xi)$ is generated by $\exp((2\pi i)^2 \langle \xi, a \rangle)$, for $a \in \Lambda$, and therefore $J_{n}^{-1}(\xi)$ is Bohr-Sommerfeld if and only if $\xi \in \frac{1}{2\pi i} \Lambda^{*}$.

\end{proof}

The equivariant structure on the line bundle $L$ induces the structure of a complex $\mathbb{T}_{\mathbb{C}}$ representation on the space of global holomorphic sections $H^{0}(M, L^{\otimes n})$, and a result of Atiyah and Danilov \cite{danilov1978geometry, guillemin2002moment} states that its weight space decomposition is indexed by the integral points in the moment polytope:
\[
H^{0}(M, L^{\otimes n}) \cong  \bigoplus_{ \lambda \in \frac{1}{2\pi}\Lambda^{*} \cap n \Delta } \mathbb{C} \lambda.
\]
Combined with Proposition \ref{BSdetermination}, this implies that we have the desired isomorphism $\Hom{ \mathcal{L}_{n}, \mathcal{B}^{cc}_{n}} \cong H^{0}(M, L^{\otimes n})$. 

There is in fact a canonical isomorphism, which we can describe explicitly. Let $M^{o} \subset M$ denote the dense open subset where the $\mathbb{T}_{\mathbb{C}}$ action is free and transitive. The proof of Proposition \ref{BSdetermination} implies that the map $\pi: (J_{n}|_{M^{o}})^{-1}(\xi) \to M^{o}$ is an isomorphism, which can be inverted to give an equivariant section $s_{\xi} : M^{o} \to Z_{n}$. By Proposition \ref{liftedequivariantaction}, there is a canonical $\mathbb{T}_{\mathbb{C}}$-equivariant isomorphism $s_{\xi}^{*}U_{n} \cong L^{\otimes n}|_{M^{o}}$. When $s_{\xi}$ is Bohr-Sommerfeld, this induces an injective map
\[
H^{0}(J_{n}^{-1}(\xi), U_{n})^{D_{n}} \to H^{0}(M^{o}, L^{\otimes n}). 
\]
By Proposition \ref{liftedequivariantaction}, the one-dimensional vector space $H^{0}(J_{n}^{-1}(\xi), U_{n})^{D_{n}}$ maps to the weight $2 \pi i \xi$ subspace of $H^{0}(M^{o}, L^{\otimes n})$. If $\xi \in -i n \Delta$, then the sections in this weight space extend to give holomorphic sections over $M$. In this way we obtain the desired isomorphism. 

\begin{theorem} \label{QuantizationTheorem}
There is a canonical $\mathbb{T}_{\mathbb{C}}$-equivariant isomorphism 
\[
\Hom{\mathcal{L}_{n}, \mathcal{B}^{cc}_{n}} \cong H^{0}(M, L^{\otimes n}),
\]
given by mapping a flat section of $U_{n}$ along ${J_{n}^{-1}(\frac{\lambda}{2\pi i})}$, for $\lambda \in \Lambda^{*} \cap 2 \pi n \Delta$, to a weight $\lambda$-section of $H^{0}(M, L^{\otimes n})$ via the isomorphism of Proposition \ref{liftedequivariantaction}.
\end{theorem}

We end this section by commenting on the close relationships between the constructions of this section and geometric quantization. First, let us consider the geometric quantization of the holomorphic symplectic manifold $(Z_{n}, \Omega_{n})$. This is defined as the space of covariantly constant sections of $(U_{n}, D_{n})$ along a foliation by Lagrangian submanifolds. If we first choose the foliation given by the fibres of $\pi: Z_{n} \to M$, then we see, as above, that geometric quantization recovers $H^{0}(M, L^{\otimes n})$. On the other hand, if we use the foliation given by the fibres of $J_{n}$, then the result of geometric quantization consists of the flat sections along Bohr-Sommerfeld fibres. By Proposition \ref{BSdetermination} and the following discussion, we see that this results in the space $H^{0}(M^{o}, L^{\otimes n})$ of \emph{meromorphic} sections of $L^{\otimes n}$. Hence, geometric quantization of $(Z_{n}, \Omega_{n})$ is not invariant under change of polarizations. Strikingly, it is the space of homomorphisms $\Hom{\mathcal{L}_{n}, \mathcal{B}_{n}^{cc}}$, defined by requiring intersection between the leaves of the foliation and the brane $\mathcal{L}_{n}$, which is the invariantly defined space. 

Second, let us consider the geometric quantization of the symplectic manifold $(M, n \omega)$. This is defined by taking the covariantly constant sections of $(L^{\otimes n}, \nabla^{\otimes n})$ along a polarization of $n \omega$. The complex structure $I$ making up the K\"{a}hler manifold defines a \emph{complex} polarization, and the vector space produced by geometric quantization is the space of holomorphic sections $H^{0}(M, L^{\otimes n})$. Since we are on a toric manifold, there is a second \emph{real} polarization which is available to us. It is given by the fibres of the moment map $n \mu: M \to \mathfrak{t}^{*}$. The quantization produced by this polarization consists of the flat sections along the Bohr-Sommerfeld fibres, which are indexed by the lattice points in $n \Delta$. 

In the present section we explained that there is a canonical isomorphism between $\Hom{\mathcal{L}_{n}, \mathcal{B}_{n}^{cc}}$ and the geometric quantization $H^{0}(M, L^{\otimes n})$ associated to the complex polarization. But there is also a canonical isomorphism between $\Hom{\mathcal{L}_{n}, \mathcal{B}_{n}^{cc}}$ and the geometric quantization associated to the real moment map. Indeed, the Bohr-Sommerfeld fibres of $J_{n}$ intersect $\mathcal{L}_{n}$ along the corresponding Bohr-Sommerfeld fibres of $n \mu$. Furthermore, by Proposition \ref{reductionprop}, it is clear that the flat sections of $D_{n}$ restrict to give flat sections of $\nabla^{\otimes n}$. The upshot is that $\Hom{\mathcal{L}_{n}, \mathcal{B}_{n}^{cc}}$ provides an intermediary with which to identify the quantizations associated to the different available polarizations.  

\subsection{The homogeneous coordinate ring} \label{algebraconst}
Given a holomorphic line bundle $L$ prequantizing a compact K\"{a}hler manifold $(M, I, \omega)$, we obtain not only a complex vector space $H^{0}(M, L)$, but also a graded algebra 
\[
\mathcal{A} = \bigoplus_{n \geq 0}H^{0}(M, L^{\otimes n}).
\]
The individual summands may be interpreted as the geometric quantizations of the rescaled symplectic forms $n \omega$. Hence, assuming independence of quantization under changes of polarization, we expect that the dimensions of these vector spaces is invariant under deformations of the complex structure for which $(M,I, \omega)$ remains K\"{a}hler. Instead, it is the product on the algebra which encodes the complex structure on $M$. In this section, we explain how this product arises from an associative operation on the cotangent bundle of $M$. 

\subsubsection{Monoidal structure on the A-model of the cotangent bundle} \label{Monoidalcotangentcat}
The real cotangent bundle of $M$ is a symplectic groupoid. This structure consists in a partial multiplication, defined by the fibrewise addition map
\[
A: T^{*}M \times_{M} T^{*}M \to T^{*}M,
\]
which makes $T^{*}M$ into a Lie groupoid over $M$. Furthermore, the canonical symplectic form $\Im(\Omega_{0})$ is \emph{multiplicative}, meaning that the graph of the multiplication map is a Lagrangian submanifold
\[ 
Gr(A) \subset (T^{*}M, -\Im(\Omega_{0})) \times (T^{*}M, -\Im(\Omega_{0})) \times (T^{*}M,  \Im(\Omega_{0})).
\] 
Another way of formulating the multiplicativity property is via symplectic reduction; the submanifold of composable arrows $T^{*}M \times_{M} T^{*}M \subset (T^{*}M, \Im(\Omega_{0})) \times (T^{*}M, \Im(\Omega_{0}))$ is coisotropic, and the addition map defines an isomorphism to its symplectic reduction. 

It is expected that the multiplication map of a symplectic groupoid induces a monoidal structure on its category of A-branes \cite{pascaleff2018poisson}. In the first place, this means that given a pair of branes $\mathcal{B}_{1}$ and $\mathcal{B}_{2}$, it is possible to `tensor' them to produce a third brane $\mathcal{B}_{1} \ast \mathcal{B}_{2}$. This product can be defined via symplectic reduction. Namely, we take the product $\mathcal{B}_{1} \times \mathcal{B}_{2} \subseteq T^{*} M \times T^{*}M$, intersect it with the coisotropic submanifold of composable arrows $T^{*}M \times_{M} T^{*}M$, and then push the resulting submanifold forward using the addition map to obtain $\mathcal{B}_{1} \ast \mathcal{B}_{2}$. When the branes are prequantized by line bundles $U_{1}$ and $U_{2}$, we obtain a prequantization $U_{1} \ast U_{2}$ of their tensor product. For this, we start by taking the exterior tensor product $U_{1} \boxtimes U_{2}$, which is supported on $\mathcal{B}_{1} \times \mathcal{B}_{2}$. The fibres of the addition map coincide with the isotropic leaves of the characteristic foliation of $T^{*}M \times_{M} T^{*}M$. As a result, the product connection on $U_{1} \boxtimes U_{2}$ is flat along this foliation; we define $U_{1} \ast U_{2}$ by taking covariant constant sections along the leaves.  

A monoidal category includes the data of an identity object, which is a left and right unit for the tensor product. For the category of A-branes on the cotangent bundle, this identity is provided by the zero section $\mathcal{L}$ equipped with the trivial local system. Indeed, $\mathcal{L}$ is a Lagrangian submanifold, and because it consists of the $0$ covectors, there are canonical isomorphisms 
\[
\mathcal{L} \ast \mathcal{B} \cong \mathcal{B} \cong \mathcal{B} \ast \mathcal{L}
\]
for any other brane $\mathcal{B}$.

The conjectural monoidal structure on the category of A-branes also defines a `tensor' product for morphisms, so that given branes $\mathcal{B}_{i}$, for $i = 1, ... , 4$, we get a $\mathbb{C}$-linear `vertical composition' map 
\[
\ast: \Hom{\mathcal{B}_{1}, \mathcal{B}_{2}} \otimes \Hom{\mathcal{B}_{3}, \mathcal{B}_{4}} \to \Hom{\mathcal{B}_{1} \ast \mathcal{B}_{3}, \mathcal{B}_{2} \ast \mathcal{B}_{4}}.  
\]
Since morphisms between coisotropic branes have not been defined mathematically, it is not yet possible to give a general definition for this product. However, for the branes that we will consider, we suggest the following definition.

 Let $\mathcal{L}_{i}$ and $\mathcal{L}_{j}$ be Lagrangian branes with trivial local systems, and let $\mathcal{B}^{cc}_{i}$ and $\mathcal{B}^{cc}_{j}$ be space-filling branes, which are prequantized by line bundles $U_{i}$ and $U_{j}$, respectively. We assume that it is possible to choose holomorphic Lagrangian foliations $\mathcal{F}_{i}, \mathcal{F}_{j}$, and $\mathcal{F}_{i,j}$ for the respective holomorphic symplectic structures $\Omega_{i}, \Omega_{j}$ and $\Omega_{i,j}$ underlying the space-filling branes $\mathcal{B}_{i}^{cc}, \mathcal{B}_{j}^{cc}$ and $\mathcal{B}_{i}^{cc} \ast \mathcal{B}_{j}^{cc}$. We assume furthermore that these foliations are compatible with the monoidal structure on the category in the following sense. Let $\Lambda_{i}$ be a leaf of $\mathcal{F}_{i}$, equipped with the local system $U_{i}|_{\Lambda_{i}}$, and let $\Lambda_{j}$ be a leaf of $\mathcal{F}_{j}$, equipped with the local system $U_{j}|_{\Lambda_{j}}$. These define branes and hence they may be tensored to define a third brane $\Lambda_{i} \ast \Lambda_{j}$, which is prequantized by the line bundle $(U_{i}|_{\Lambda_{i}}) \ast (U_{j}|_{\Lambda_{j}})$. We assume that this is a leaf of the foliation $\mathcal{F}_{i,j}$. Finally, we assume that $\Lambda_{i} \ast \Lambda_{j}$ intersects $\mathcal{L}_{i} \ast \mathcal{L}_{j}$, if $\Lambda_{i}$ and $\Lambda_{j}$ respectively intersect $\mathcal{L}_{i}$ and $\mathcal{L}_{j}$. 
 
Recall from Equation \ref{proposalforquant} that the morphisms in $\Hom{\mathcal{L}_{i}, \mathcal{B}_{i}^{cc}}$ are given by flat sections of $U_{i}$ which are supported on leaves of $\mathcal{F}_{i}$ that intersect $\mathcal{L}_{i}$, and similarly for $\Hom{\mathcal{L}_{j}, \mathcal{B}_{j}^{cc}}$ and $\Hom{\mathcal{L}_{i} \ast \mathcal{L}_{j}, \mathcal{B}_{i}^{cc} \ast \mathcal{B}_{j}}$. Now let $t_{i} \in \Hom{\mathcal{L}_{i}, \mathcal{B}_{i}^{cc}}$, which is a flat section supported on a leaf $\Lambda_{i}$, and let $t_{j} \in \Hom{\mathcal{L}_{j}, \mathcal{B}_{j}^{cc}}$, which is a flat section supported on a leaf $\Lambda_{j}$. By taking their exterior tensor product, we obtain a flat section $t_{i} \ast t_{j}$ of $U_{i} \ast U_{j}$, supported on $\Lambda_{i} \ast \Lambda_{j}$. Therefore, this defines a morphism 
\[
t_{i} \ast t_{j} \in \Hom{\mathcal{L}_{i} \ast \mathcal{L}_{j}, \mathcal{B}_{i}^{cc} \ast \mathcal{B}_{j}}. 
\]
 We take this as our proposal for the tensor product of the morphisms $t_{i}$ and $t_{j}$. 

%
%

\subsubsection{Algebras from monoid objects} \label{Algfrommonoid}
To construct an associative algebra from a monoidal category, we need the data of a monoid object. This is an object $\mathcal{B}$, equipped with a `product' morphism $\mu : \mathcal{B} \ast \mathcal{B} \to \mathcal{B}$ and a `unit' morphism $\eta: \mathcal{L} \to \mathcal{B}$ from the identity object, satisfying the usual associativity and unit equations. Given such an object, we obtain a map 
\[
\Hom{\mathcal{L}, \mathcal{B}} \otimes \Hom{\mathcal{L}, \mathcal{B}} \to \Hom{\mathcal{L} \ast \mathcal{L}, \mathcal{B} \ast \mathcal{B}} \to \Hom{\mathcal{L}, \mathcal{B}},
\]
by composing the tensor product, post-composition with $\mu$, and pre-composition with the canonical isomorphism $\mathcal{L} \ast \mathcal{L} \cong \mathcal{L}$. This defines the structure of an associative algebra on $\Hom{\mathcal{L}, \mathcal{B}}$, with multiplicative identity given by $\eta$.


The homogeneous coordinate ring $\mathcal{A}$ is a $\mathbb{Z} -$\emph{graded} algebra. To construct it from the category of A-branes, we need the data of a $\mathbb{Z} -$\emph{graded} monoid. This consists of a family $\{ \mathcal{B}_{n} \}_{n \in \mathbb{Z}}$ of objects, equipped with `product maps'
\begin{equation} \label{braneproduct}
\mu: \mathcal{B}_{n} \ast \mathcal{B}_{m} \to \mathcal{B}_{n + m},
\end{equation}
and a unit morphism $\eta: \mathcal{L} \to \mathcal{B}_{0}$, all subject to the usual associativity and unit equations. Given this data, the direct sum 
\[
\bigoplus_{n \geq 0} \Hom{\mathcal{L}, \mathcal{B}_{n}}
\]
is equipped with the structure of an associative $\mathbb{Z}-$graded algebra with multiplicative identity given by $\eta$. 

\subsubsection{Constructing the graded homogenous coordinate ring} \label{Congradedhomcoordring}

Let $(M, I, \omega)$ be a prequantized compact connected toric K\"{a}hler structure, as in Section \ref{toricsection}, and let $(\mathcal{L}_{n}, \mathcal{B}_{n}^{cc})_{n \in \mathbb{Z}}$ be the collection of branes constructed in Proposition \ref{reductionprop}. In order to view these as branes in a common real cotangent bundle, we use the isomorphism of Equation \ref{magdef}. Under this isomorphism, the branes $\mathcal{L}_{n}$ are mapped to the zero section equipped with the trivial local system. Therefore they all coincide with the monoidal unit $\mathcal{L}$, and we will use both notations interchangeably. In this section, we will show that the branes $\mathcal{B}_{n}^{cc}$ form a $\mathbb{Z}-$graded monoid object as above, for which the multiplication map $\mu$ of Equation \ref{braneproduct} is an isomorphism. We will then construct a graded algebra following the construction outlined in the previous section, and show that it coincides with the homogeneous coordinate ring $\mathcal{A}$.

The space filling branes $\mathcal{B}^{cc}_{n}$ provide prequantized holomorphic symplectic structures $(Z_{n}, \Omega_{n}, U_{n}, D_{n})$. Their tensor products $\mathcal{B}^{cc}_{n} \ast \mathcal{B}^{cc}_{m}$ correspond to the holomorphic symplectic reductions of the coisotropic submanifolds $Z_{n} \times_{X} Z_{m} \subset (Z_{n}, \Omega_{n}) \times (Z_{m}, \Omega_{m})$. Hence, the collection of maps of Equation \ref{braneproduct} is given by an isomorphism between the symplectic reduction of $Z_{n} \times_{X} Z_{m}$ and $(Z_{n+m}, \Omega_{n + m})$, for each pair $(n, m) \in \mathbb{Z}^2$. Assembling all of these maps together, we obtain the structure of a holomorphic symplectic groupoid on the disjoint union
\[
(\mathcal{Z}, \Omega) := \sqcup_{n \in \mathbb{\mathbb{Z}}} (Z_{n}, \Omega_{n}). 
\]
The associativity equations for the graded monoid are encoded by the associativity of the groupoid product. 

The Lagrangian branes $\mathcal{L}_{n} \subset Z_{n}$ allow us to identify $Z_{n}$ as a brane $\mathcal{B}_{n}^{cc}$ in the cotangent bundle $(T^{*}M, \Im(\Omega_0))$. For the multiplicative structure on $\{{Z}_n\}$ in the groupoid $\mathcal{Z}$ to be compatible with the convolution product of the corresponding branes $\{\mathcal{B}_{n}^{cc}\}$ in $T^*M$,  
we must require that under the groupoid product we have $\mathcal{L}_{n} \ast \mathcal{L}_{m} = \mathcal{L}_{n+m}$.

We can upgrade this to include the prequantizations by giving a flat isomorphism between $U_{n} \boxtimes U_{m}$ and the pullback of $U_{n+m}$ to $Z_{n} \times_{X} Z_{m} $, for each pair $(n, m) \in \mathbb{Z}^2$. This data is in turn equivalent to a multiplicative structure on the corresponding prequantization of $(\mathcal{Z}, \Omega)$
\[
(\mathcal{U}, D) := \sqcup_{n \in \mathbb{\mathbb{Z}}} (U_{n}, D_{n}). 
\]
Recall that a multiplicative structure on a line bundle over a Lie groupoid is defined to be a flat trivialization
\[
\Theta \in H^{0}(\mathcal{Z}^{(2)}, \delta \mathcal{U}),
\]
where $\mathcal{Z}^{(2)}$ is the space of composable arrows, and $\delta \mathcal{U} = p_{1}^{*}\mathcal{U}^{*} \otimes p_{2}^{*}\mathcal{U}^{*} \otimes m^{*}\mathcal{U}$ as a line bundle with connection. In the last expression, $m$ is the product on $\mathcal{Z}$, and $p_{1}, p_{2}$ are the two projection maps. This trivialization is furthermore required to satisfy a cocycle condition on the space $\mathcal{Z}^{(3)}$ of triples of composable arrows. Namely, let $p_{i} : \mathcal{Z}^{(3)} \to \mathcal{Z}^{(2)}$, for $i = 1,...,4$, denote the four maps given by projections and multiplications. Then there is a canonical flat isomorphism 
\[
\delta \delta \mathcal{U} := p_{1}^{*}(\delta \mathcal{U}) \otimes p_{2}^{*}(\delta \mathcal{U}^{*}) \otimes p_{3}^{*}(\delta \mathcal{U}) \otimes p_{4}(\delta \mathcal{U}^{*}) \cong (\mathcal{O}, d). 
\]
The cocycle condition states that 
\[
\delta(\Theta) := p_{1}^{*}\Theta \otimes p_{2}^{*}\Theta^{-1} \otimes p_{3}^{*}\Theta \otimes p_{4}^{*}\Theta^{-1} = 1.
\]
This condition encodes the associativity equations for the graded monoid. We refer to the resulting structure $(\mathcal{Z}, \Omega, \mathcal{U}, D, \Theta)$ as a prequantized symplectic groupoid. 

The unit of the graded monoid is a morphism $\eta \in \Hom{\mathcal{L}, \mathcal{B}_{0}^{cc}}$. The holomorphic symplectic manifold underlying $\mathcal{B}_{0}^{cc}$ is the holomorphic cotangent bundle, which is equipped with the holomorphic moment map $J_{0}$ constructed in Section \ref{toricsection}. Then, as in Equation \ref{proposalforquant}, the morphism $\eta$ is given by a flat section of $(U_{0}, D_{0})$, which is supported on a fibre of $J_{0}$ which intersects $\mathcal{L}_{0}$. In fact, $\mathcal{L}_{0}$ is itself a fibre of $J_{0}$, and Proposition \ref{reductionprop} implies that $(U_{0}, D_{0})|_{\mathcal{L}_{0}} \cong (\mathcal{O}, d)$. Therefore, the flat section $\eta$ may be given by the constant section $1$ of $\mathcal{O}$. The existence of such a unit morphism $\eta$ is actually a direct consequence of the structure of a prequantized symplectic groupoid. 

The upshot of the present discussion is therefore that the structure of a graded monoid for the branes $\{ \mathcal{B}_{n}^{cc} \}_{n \in \mathbb{Z}}$ is equivalent to the structure of a holomorphic symplectic groupoid with prequantizing multiplicative line bundle $(\mathcal{Z}, \Omega, \mathcal{U}, D, \Theta)$, equipped with brane sections $\mathcal{L}_{n} \subset Z_{n}$ satisfying $\mathcal{L}_{n} \ast \mathcal{L}_{m} = \mathcal{L}_{n+m}$ for all $(n,m) \in \mathbb{Z}^2$. We construct this in the following proposition. 

\begin{proposition} \label{multstruct1}
The holomorphic symplectic manifold $(\mathcal{Z}, \Omega)$ is equipped with the structure of a holomorphic symplectic groupoid. With respect to the product on this groupoid, the Lagrangian submanifolds $\mathcal{L}_{n} \subset \mathcal{Z}_{n}$ satisfy
\[
\mathcal{L}_{n} \ast \mathcal{L}_{m} = \mathcal{L}_{n+m}
\]
for all $(n,m) \in \mathbb{Z}$. Furthermore, the prequantum line bundle $(\mathcal{U}, D)$ is equipped with a multiplicative structure $\Theta$. Using the identification $U_{n} \cong \pi^{*}L^{\otimes n}$ of Proposition \ref{reductionprop}, we have an isomorphism 
\[
\delta \mathcal{U} \cong \pi^{*}((L^{*})^{\otimes n} \otimes (L^{*})^{\otimes m} \otimes L^{\otimes (n + m)}),
\]
over $Z_{n} \times_{X} Z_{m}$, where $\pi$ is the projection to $M$. The multiplicative structure $\Theta|_{Z_{n} \times_{M} Z_{m}}$ is given by the canonical trivialisation of the line bundle on the right hand side.

Finally, the Hamiltonian torus action on $\mathcal{Z}$ is by groupoid automorphisms, and the individual moment maps $J_{n}$ assemble into a moment map $J: \mathcal{Z} \to \mathfrak{t}^{*}_{\mathbb{C}}$ defining a groupoid homomorphism, where $\mathfrak{t}^{*}_{\mathbb{C}}$ is equipped with its additive group structure. 
\end{proposition}
\begin{proof}
We prove this theorem via symplectic reduction, as in Section \ref{reductionsection}. Let $P_{L}$ be the principal $\mathbb{C}^{*}$-bundle associated to $L$. It's cotangent bundle $(T^{*}P_{L}, \Omega_{0})$ is a holomorphic symplectic groupoid. Furthermore, the prequantization $(\mathcal{O}, d - 2\pi i \alpha_{0})$ is equipped with a natural multiplicative structure. Namely, the form $\alpha_{0}$ is multiplicative, meaning that $m^{*}\alpha_{0} = p_{1}^{*}\alpha_{0} + p_{2}^{*}\alpha_{0}$. Therefore, the induced line bundle with connection on $(T^{*}P_{L})^{(2)}$ is the trivial bundle with trivial flat connection, and there is a multiplicative structure given by $\Theta_{0} = 1$. 

Let $\mathcal{E}$ be the vector field on $P_{L}$ which generates the scaling action. Pairing with this vector field defines the moment map for the Hamiltonian $\mathbb{C}^{*}$-action on $T^{*}P_{L}$. This defines a groupoid homomorphism $\mathcal{E}: T^{*}P_{L} \to \mathbb{C}$, where $\mathbb{C}$ is equipped with it's additive group structure. Furthermore, the $\mathbb{C}^{*}$-action on $T^{*}P_{L}$ is by groupoid automorphisms. Therefore, reducing $T^{*}P_{L}$ at a discrete subgroup of $\mathbb{C}$ produces a holomorphic symplectic groupoid. Hence 
\[
\mathcal{Z} = \mathcal{E}^{-1}(\tfrac{1}{2 \pi} \mathbb{Z})/\mathbb{C}^{*},
\]
is a holomorphic symplectic groupoid. The multiplicative structure $\Theta_{0}$ descends to give a multiplicative structure $\Theta$ on $\mathcal{U}$ as long as it is $\mathbb{C}^{*}$-invariant. But this follows immediately from the fact that $\Theta_{0}$ is flat. 

The submanifolds $\mathcal{L}_{n}$ are obtained by reducing the conormal bundle $C_{S}$. Since $C_{S}$ is a subgroupoid of $T^{*}P_{L}$, the union $\sqcup_{n \in \mathbb{Z}} \mathcal{L}_{n}$ forms a subgroupoid of $\mathcal{Z}$. Hence $\mathcal{L}_{n} \ast \mathcal{L}_{m} = \mathcal{L}_{n+m}$ holds for all $(n,m)$. 

In order to identify $\Theta$, we pullback the isomorphism $U_{n} \cong \pi^{*}L^{\otimes n}$ to $\mathcal{E}^{-1}(\frac{n}{2\pi})$, where it gives an isomorphism $\mathcal{O}|_{\mathcal{E}^{-1}(\frac{n}{2\pi})} \cong \pi^{*}L^{\otimes n}$. This isomorphism is realized by the canonical section $s_{n}$ which sends a point $\alpha \in \mathcal{E}^{-1}(\frac{n}{2\pi})$ lying above $\lambda \in P_{L}$ to the point $\lambda^{\otimes n} \in L^{\otimes n}$. Under this identification, the multiplicative structure $\Theta_{0}$ restricted to $\mathcal{E}^{-1}(\frac{n}{2\pi}) \times_{P_{L}} \mathcal{E}^{-1}(\frac{m}{2\pi})$ is identified with the section $s_{n}^{-1} \otimes s_{m}^{-1} \otimes s_{n+m}$, which is simply the canonical trivialisation. 

Finally, recall that the Hamiltonian $\mathbb{T}_{\mathbb{C}}$-action on the spaces $Z_{n}$ is induced via symplectic reduction from a Hamiltonian action on $T^{*}P_{L}$. It is straightforward to check that this action preserves the groupoid structure, and that the moment map defines a groupoid homomorphism. It then follows by reduction that the same is true for the action on $\mathcal{Z}$. 
\end{proof}

\begin{remark}
Proposition \ref{multstruct1} shows that the groupoid structure on $\mathcal{Z}$ arises via symplectic reduction from the groupoid structure on $T^{*}P_{L}$. This may be interpreted as the statement that the graded monoid $\{\mathcal{B}_{n}^{cc} \}_{n \in \mathbb{Z}}$ arises from a single $\mathbb{C}^{*}$-equivariant monoid space-filling brane $\mathcal{B}^{cc}$ on $T^{*}P_{L}$. Similarly, the family of Lagrangian branes $\{ \mathcal{L}_{n} \}_{n \in \mathbb{Z}}$ arises from a single $S^{1}$-equivariant Lagrangian brane $C_{S}$ on $T^{*}P_{L}$ (this is the conormal bundle from Proposition \ref{reductionprop}). This brane is a comonoid, since there is an isomorphism 
\[
C_{S} \cong C_{S} \ast C_{S}. 
\]
Hence, we also expect to obtain an algebra $\Hom{C_{S}, \mathcal{B}^{cc}}$, which is $\mathbb{Z}-$graded by the equivariant $S^{1}$-structure. 
\end{remark}

In addition to constructing the multiplication maps for the graded monoid $\{ \mathcal{B}_{n}^{cc} \}_{n \in \mathbb{Z}}$, Proposition \ref{multstruct1} also shows that there is a groupoid homormophism $J: \mathcal{Z} \to \mathfrak{t}^{*}_{\mathbb{C}}$ given by the moment maps constructed in Section \ref{toricsection}. As a result, the holomorphic Lagrangian foliation given by the fibres of $J$ is multiplicative. Indeed, given the fibres $J^{-1}_{n}(\xi)$ and $J^{-1}_{m}(\eta)$, for $\xi, \eta \in \mathfrak{t}^{*}_{\mathbb{C}}$, we have the identity 
\[
J^{-1}_{n}(\xi) \ast J^{-1}_{m}(\eta) = J^{-1}_{n + m}(\xi + \eta). 
\]

We can now apply the proposal, outlined above, for taking the tensor product of morphisms. Let $J_{n}^{-1}(\xi)$ and $J_{m}^{-1}(\eta)$ be two Bohr-Sommerfeld fibres which intersect $\mathcal{L}_{n}$ and $\mathcal{L}_{m}$, respectively. 
Let $t_{n , \xi} \in H^{0}(J_{n}^{-1}(\xi), U_{n})^{D_{n}}$ and $t_{m, \eta} \in H^{0}(J_{m}^{-1}(\eta), U_{m})^{D_{m}}$ be flat sections defined on the respective Bohr-Sommerfeld fibres. These define morphisms
\[
t_{n , \xi} \in \Hom{ \mathcal{L}_{n}, \mathcal{B}_{n}^{cc}}, \qquad t_{m , \eta} \in \Hom{\mathcal{L}_{m}, \mathcal{B}_{m}^{cc}}.
\]
Their tensor product $t_{n , \xi} \ast t_{m, \eta} \in \Hom{\mathcal{L}_{n+m}, \mathcal{B}_{n+m}^{cc}}$ is defined to be 
\[
\Theta(t_{n , \xi} \boxtimes t_{m, \eta}) \in H^{0}(J_{n + m}^{-1}(\xi + \eta), U_{n + m})^{D_{n + m}},
\]
which is a flat section supported on the Bohr-Sommerfeld fibre $J^{-1}_{n + m}(\xi + \eta)$. This defines the structure of a graded algebra. The following theorem identifies it with the homogeneous coordinate ring.

\begin{theorem}
The additive structure on the real cotangent bundle $(T^{*}M, \Im(\Omega_{0}))$ induces the structure of an algebra on 
\[
 \bigoplus_{n \geq 0} \Hom{\mathcal{L}_{n}, \mathcal{B}^{cc}_{n}}.
\]
Using the identification from Theorem \ref{QuantizationTheorem}, this algebra is canonically isomorphic to the homogeneous coordinate ring
\[
\mathcal{A} = \bigoplus_{n \geq 0}H^{0}(M, L^{\otimes n}).
\]
\end{theorem}
\begin{proof}
By Theorem \ref{QuantizationTheorem}, we have canonical isomorphisms $\Hom{\mathcal{L}_{n}, \mathcal{B}^{cc}_{n}} \cong H^{0}(M, L^{\otimes n})$. Hence, we only need to show that under these isomorphisms the tensor product $\ast$ defined above agrees with the tensor product of sections defining the algebra $\mathcal{A}$. 

Recall from the proof of Theorem \ref{QuantizationTheorem} the equivariant sections $s_{n , \xi} : M^{o} \to Z_{n}$ which map into  $J_{n}^{-1}(\xi)$, for $\xi \in (-i n \Delta) \cap (\frac{1}{2\pi i} \Lambda^{*})$. The isomorphism of Theorem \ref{QuantizationTheorem} is constructed by pulling back a flat section $t_{n , \xi} \in H^{0}(J_{n}^{-1}(\xi), U_{n})^{D_{n}}$ along the map $s_{n, \xi}$, and then using the isomorphism of Proposition \ref{liftedequivariantaction} to identify it as a section of $L^{\otimes n}$. Proposition \ref{multstruct1} implies that under this isomorphism, the multiplicative structure $\Theta$ is identified with the tensor product of line bundles. Therefore, for morphisms $t_{n , \xi} $ and $t_{m, \eta}$, we obtain the identification
\[
t_{n , \xi} \ast t_{m, \eta}  = \Theta(t_{n , \xi} \boxtimes t_{m, \eta}) = t_{n , \xi} \otimes t_{m, \eta}.
\]
\end{proof}

\section{Noncommutative deformation via brane quantization}\label{ncsection}
Let $(M, I, \omega, L)$ be a prequantized K\"{a}hler manifold, with homogeneous coordinate ring $\mathcal{A} =  \bigoplus_{n \geq 0}H^{0}(M, L^{\otimes n})$. In the usual approach, in contrast to the approach of the previous sections, the complex polarization is used to view the prequantum line bundles $L^{\otimes n}$ as objects in the category of vector bundles over the complex manifold $(M, I)$. In other words, the line bundles are viewed as B-branes. From this point of view, the graded components of the homogeneous coordinate ring are given by spaces of homomorphisms 
\[
H^{0}(M, L^{\otimes n}) = \Hom{\mathcal{O}, L^{\otimes n}}. 
\]
To multiply elements $a \in \Hom{\mathcal{O}, L^{\otimes n}}$ and $b \in \Hom{\mathcal{O}, L^{\otimes m}}$, we take their composition 
\[
a \circ b \in \Hom{\mathcal{O}, L^{\otimes (n +m)}},
\]
viewing $a$ as an element in $\Hom{L^{\otimes m}, L^{\otimes (n+m)}}$ by using the tensor product with $L^{\otimes m}$. Therefore, the results of the previous sections may be understood as giving a reinterpretation of a standard construction in the B-model of a complex manifold in terms of a construction in the A-model of its cotangent bundle. More precisely, the space of morphisms $\Hom{L^{\otimes n}, L^{\otimes m}}$ is identified with a space of morphisms between A-branes $\Hom{\mathcal{L}, \mathcal{B}_{m-n}^{cc}}$, and the composition of morphisms is identified with the `vertical composition' defined by the monoidal structure on the A-branes. 

In this section, we will deform the homogeneous coordinate ring by viewing the complex polarization $I$ as a generalized complex (GC) structure, which can be deformed by turning on a holomorphic Poisson structure $\sigma$ in the following way:
\begin{equation}\label{hbarfamilyj}
\JJ_\hbar = \begin{pmatrix}-I&Q_\hbar \\0&I^*\end{pmatrix},
\end{equation}
where $Q_\hbar =-4\Im(\hbar\sigma)$.
Just as in the case of a complex manifold, where the homogeneous coordinate ring captures the data of the polarization, the resulting non-commutative algebra $\mathcal{A}_{\hbar}$ will quantize the holomorphic Poisson structure $\sigma$. 

In order to implement this procedure, we need to deform the line bundles $L^{\otimes n}$ into branes for the family of GC structures $\mathbb{J}_{\hbar}$, and furthermore carry out computations involving homomorphisms between branes. In other words, we need to be able to work in a category of branes for $\mathbb{J}_{\hbar}$.

\subsection{Morphisms between generalized complex branes} \label{proposal}

A generalized complex (GC) structure $\mathbb{J}$ is expected to have a category of branes, simultaneously generalizing the category of A-branes of a symplectic manifold and the category of B-branes (vector bundles/coherent sheaves) of a complex manifold. The objects of this category were defined by the second author~\cite{MR2811595}, but the definition of morphisms remains incomplete at the time of writing. Inspired by the relationship mentioned above between the B-model of $(M,I)$ and the A-model of its cotangent bundle, we propose to define this category by making use of a monoidal category of A-branes which is naturally associated to a generalized complex structure. Namely, a GC structure $\mathbb{J}$ has an underlying Poisson structure $Q$, which we assume is integrable to a symplectic groupoid $(\mathcal{G}, \omega)$ (global integrations may be obstructed \cite{crainic2003integrability}, but integrations to local symplectic groupoids are always guaranteed to exist \cite{coste1987groupoides,MR2900786,cabrera2018local}). The product on this groupoid is expected to endow its category of A-branes with a monoidal structure, with identity object given by the Lagrangian submanifold $\mathcal{L}$ of identity arrows \cite{pascaleff2018poisson}. The GC structure $\mathbb{J}$ then endows the anchor map $(t,s): \mathcal{G} \to M \times M$ with the structure of a generalized holomorphic map (in general we must also modify the symplectic groupoid using a B-field, but this is not required in our situation) \cite{crainic2011}. As a result, given a pair of GC branes $B_{1}$ and $B_{2}$, we may construct an A-brane $\mathcal{B}_{2,1}$ in the groupoid whose support consists in the arrows with source $B_{1}$ and target $B_{2}$:
\[
\mathcal{B}_{2,1} = t^{-1}(B_{2}) \cap s^{-1}(B_{1}). 
\]
Our proposal for the morphisms is given by 
\[
\Hom[\mathbb{J}]{B_{1}, B_{2}} := \Hom[\mathcal{G}]{\mathcal{L}, \mathcal{B}_{2,1}},
\]
with composition provided by the vertical composition of the monoidal category. \emph{In this way, we reduce the problem of constructing the category of generalized complex branes to the problem of computing homomorphisms from Lagrangian to coisotropic A-branes in the symplectic groupoid.}

We can construct a $\mathbb{Z}$-graded algebra from the data of a GC brane $B_{0}$ and an auto-equivalence $T$ of the category. Namely, by successively applying the functor to $B_{0}$ we obtain a sequence of branes $B_{n} = T^{n}(B_{0})$. The graded vector space underlying the algebra is defined to be 
\[
\mathcal{A} := \bigoplus_{n \geq 0} \Hom[\mathbb{J}]{B_{0}, B_{n}} = \bigoplus_{n \geq 0} \Hom[\mathcal{G}]{\mathcal{L}, \mathcal{B}_{n,0}},
\]
and the product is given by composition, using the functor $T^{m}$ to identify $\Hom[\mathbb{J}]{B_{0},B_{n}} \cong \Hom[\mathbb{J}]{B_{m}, B_{m+n}}$.

\subsection{R-matrices and toric Poisson structures} \label{Rmatrix}
Let $(M, I)$ be a toric manifold, and denote the infinitesimal torus action by $V: \mathfrak{t}_{\mathbb{C}} \to H^{0}(M, \mathcal{T}_{M})$, using the notation of Section \ref{toricsection}. A \emph{toric Poisson structure} $\sigma \in H^{0}(M, \wedge^{2} \mathcal{T}_{M})$ is a holomorphic Poisson structure on $M$ which is invariant under the torus action, meaning that $L_{V_{a}}(\sigma) = 0$ for all $a \in \mathfrak{t}_{\mathbb{C}} $. Such Poisson structures are always defined by triangular R-matrices for $\mathfrak{t}_{\mathbb{C}}$. 

A triangular R-matrix for $\mathfrak{t}_{\mathbb{C}}$ is defined to be an element $C \in \wedge^{2}\mathfrak{t}_{\mathbb{C}}$. Given such an R-matrix, we define a toric Poisson structure by pushing $C$ forward using the action map 
\[
\sigma_{C} := \wedge^2V(C) \in H^{0}(M, \wedge^{2} \mathcal{T}_{M}).
\]
Choose a basis $\{e_{i} \}$ of $\mathfrak{t}_{\mathbb{C}}$, as in Section \ref{toricsection}. Then the R-matrix may be expressed as 
\[
C = \tfrac{1}{2}\sum_{i,j} C_{ij} e_{i} \wedge e_{j},
\]
and the Poisson structure as
\begin{equation} \label{holPoisson}
\sigma_{C} = \tfrac{1}{2}\sum_{i,j} C_{ij} V_{i} \wedge V_{j}.
\end{equation}
This Poisson structure has the following decomposition into real and imaginary parts 
\[
\sigma_{C} = -\tfrac{1}{4}(IQ + iQ),
\]
where $Q$ is a real Poisson structure. By decomposing the R-matrix into real and imaginary parts $C = A + iB$, we have the following explicit form for $Q$
\[
Q = \tfrac{1}{2}\sum_{ij}(B_{ij}X_{i} \wedge X_{j} - 2A_{ij}IX_{i} \wedge X_{j} - B_{ij}IX_{i} \wedge IX_{j} ),
\]
where $X_{i}$ are the infinitesimal generators of the $\mathbb{T}$-action, defined in Section \ref{toricsection}.

The holomorphic Poisson structure $\sigma_{C}$ defines a family of generalized complex structures $\JJ_\hbar$, as in~\eqref{hbarfamilyj}.
 We may absorb the factor of $\hbar$ into the R-matrix $C$. Hence, in the remainder of this paper we will simply set $\hbar = 1$ and denote the corresponding GC structure by  
\begin{equation*}
\JJ_C = \begin{pmatrix}-I&Q \\0&I^*\end{pmatrix}.
\end{equation*}
\subsection{Branes on toric Poisson varieties} \label{branessection}
A prequantum line bundle $(L, h, \nabla)$ on a toric K\"{a}hler manifold defines both a sequence of B-branes for $(M, I)$ as well as an autoequivalence of the B-model category. The sequence of branes is given by the tensor powers $(L^{\otimes n}, h^{n}, \nabla^{\otimes n})$ and the autoequivalence is given by the operation of tensoring with $L$. In this section, we explain how these are deformed when we turn on the toric Poisson structure $\sigma_{C}$. 

To begin, we recall the definition of a space filling brane from \cite{MR2811595, gualtieri2010branes}. 
\begin{definition}
A space filling brane of $\mathbb{J}_{C}$ is given by a unitary line bundle with connection $(L, h, \nabla)$, whose curvature $-2\pi i F$ satisfies the following equation 
\begin{equation} \label{starequation}
FI + I^{*}F + FQF = 0. 
\end{equation}

\end{definition}

When $Q = 0$, this condition says that $F$ is a closed form of type $(1,1)$, and hence that $(L, \nabla^{0,1})$ defines a holomorphic line bundle. In the general case, this equation implies that $J = I + QF$ is a new complex structure with respect to which $ -\frac{1}{4}(JQ + iQ)$ is a holomorphic Poisson structure. Note that Equation \ref{starequation} is always satisfied by $F = 0$. This is the curvature of the canonical coisotropic brane $B_{0}$, which is defined by the trivial flat Hermitian bundle $(\mathcal{O}, h = 1, \nabla = d)$. 

\subsubsection{R-matrix deformation of space-filling branes}
Let $(M, I, \omega, (L, h, \nabla))$ be a prequantized toric K\"{a}hler structure, whose line bundle is equipped with a $\mathbb{T}_{\mathbb{C}}$-equivariant structure, as in Section \ref{toricsection}. The operation of tensor product with $(L, h, \nabla)$ defines an autoequivalence of the category of branes for $I$. Following \cite{gualtieri2010branes}, we now explain how the $R$-matrix $C$ may be used to deform this into an autoequivalence of the category of GC branes for $\mathbb{J}_{C}$. 

\begin{enumerate}
\item Let $\hat{V} : \mathfrak{t}_{\mathbb{C}} \to H^{0}(tot(L), \mathcal{T}_{tot(L)})$ denote the lift of the infinitesimal torus action which is induced by the equivariant structure of the prequantum line bundle. The R-matrix $C$ then defines a holomorphic Poisson structure $\hat{\sigma}_{C} = \wedge^2 \hat{V}(C)$ on the total space of $L$ which is invariant under $\mathbb{T}_{\mathbb{C}} \times \mathbb{C}^{*}$ and which reduces via the $\mathbb{C}^{*}$-action to $\sigma_{C}$ on $M$. In other words, $(L, \hat{\sigma}_{C})$ defines a \emph{Poisson module} for $\sigma_{C}$.

\item Combining the data of the Poisson module with the metric $h$, we obtain an \emph{infinitesimal Courant symmetry} of $\mathbb{J}_{C}$ in the following way. First, viewing $h$ as a function on $L$, we obtain the K\"{a}hler form as $p^{*}\omega = \frac{1}{2 \pi i } \partial \overline{\partial} \log h$, where $p: L \to M$. Second, taking the Hamiltonian vector field of $\log h$ with respect to the real Poisson structure $\hat{Q} = -4 \Im(\hat{\sigma}_{C}) $, we obtain 
\begin{equation}\label{hatw}
\hat{W} := \hat{Q}(\tfrac{1}{4\pi} d\log h) =- \tfrac{1}{\pi} \Im(\hat{\sigma}_{C}(\partial \log h)) = 2 \Im(\hat{V} C p^{*}\mu),
\end{equation}
which is $\mathbb{C}^{*}$-invariant and covers the vector field 
\[
W = 2 \Im(VC\mu) = -\sum_{ij}\mu_{i}(B_{ij}IX_{j} + A_{ij}X_{j})
\]
on $M$. Because $W$ is obtained from a $\hat{Q}$-Hamiltonian vector field, it is Poisson with respect to $Q$. Furthermore, it satisfies the following equation 
\[
\mathcal{L}_{W}(I) = 4 \Re( VC \overline{\partial}\mu) = -4 \Im( \sigma_{C} \omega) = Q\omega. 
\]
The pair $(W, \omega)$ defines an infinitesimal Courant symmetry in the following sense: the operator 
\[
\begin{pmatrix} \mathcal{L}_{W} & 0 \\ \omega & \mathcal{L}_{W}  \end{pmatrix}
\]
which encodes the Lie derivative action of $W$ and the B-field action of $\omega$ on sections of $TM \oplus T^{*}M$ commutes with $\mathbb{J}_{C}$.

\item The infinitesimal symmetry $(W, \omega)$ may be integrated to a \emph{global Courant symmetry} of $\mathbb{J}_{C}$. First, the flow of the vector field $W$ defines a $1$-parameter family of diffeomorphisms $\varphi_{t}$. Second, by averaging the K\"{a}hler form over the flow of $W$, we obtain the following family of $B$-field gauge symmetries:
\[
F_{t} = \int_{0}^{t} (\varphi_{s}^{*} \omega) ds. 
\]
The data $(\varphi_{t}, F_{t})$ defines a $1$-parameter family of global Courant symmetries, which acts on solutions to Equation \ref{starequation} by sending a closed $2$-form $F$ to 
\[
F_{t} + \varphi_{t}^{*}(F). 
\]
In particular, by applying this to $F=0$, we see that $F_{t}$ solves Equation \ref{starequation} for all $t$. 

\item The global Courant symmetry at $t = 1$ admits a prequantization. Namely, the form $- 2 \pi i F_{1}$ is the curvature of the following averaged connection on $L$:
\[
\overline{ \nabla} = \int_{0}^{1} (\varphi_{t}^{*}  \nabla) dt = \nabla - 2\pi i  \int_{0}^{1} \iota_{W}F_{t} dt . 
\]
We denote by $T_{C}$ the action of $(\varphi_{1}, L, h, \overline{\nabla})$ on branes for $\mathbb{J}_{C}$, which sends a brane $B$ to 
\begin{equation} \label{tensoraction}
T_{C}(B) := (L, h, \overline{ \nabla}) \otimes \varphi_{1}^{*}(B). 
\end{equation}
This is expected to extend to an auto-equivalence of the category of branes for $\mathbb{J}_{C}$, which deforms the equivalence given by tensoring by $L$. 
\end{enumerate}

By successively applying $T_{C}$ to the canonical coisotropic $B_{0}$, we obtain the following sequence of branes 
\[
B_{n} \coloneqq T^{n}_{C}(B_{0}) = (L^{\otimes n}, h^{n}, \overline{ \nabla^{\otimes n}}),
\]
where 
\[
\overline{ \nabla^{\otimes n}} = \frac{1}{n} \int_{0}^{n} (\varphi_{t}^{*}  \nabla^{\otimes n}) dt = \nabla^{\otimes n} - 2\pi i  \int_{0}^{n} \iota_{W}F_{t} dt .
\]
These branes deform the sequence of prequantum line bundles $(L^{\otimes n}, h^{n}, \nabla^{\otimes n})$. Note also that 
\begin{equation} \label{newcomplexstructures}
I_{n} := \varphi_{n}^{*}(I) = I + QF_{n}
\end{equation}
defines a new complex structure, with respect to which $\sigma_{n} =  -\frac{1}{4}(I_{n}Q + iQ)$ is holomorphic Poisson. 

\subsection{The symplectic groupoid and its monoidal category of branes} \label{XuConstruction}
In this section we construct the symplectic groupoid of $\sigma_{C}$, which will later be used to compute morphisms between the branes constructed in Section \ref{branessection}. We follow the approach developed in the work of Xu \cite{xu1993poisson}. 

The starting point is the observation that $\sigma_{C}$ may be obtained by reduction starting from a Poisson structure on the product $\mathbb{T}_{\mathbb{C}} \times M$. This Poisson structure is given by the product $C^{R} \times 0$, where $C^{R}$ is the right-invariant holomorphic Poisson structure on $\mathbb{T}_{\mathbb{C}}$ induced by the R-matrix $C$, and $0$ is the trivial Poisson structure on $M$. In terms of coordinates $(w_{i})$ defined by the isomorphism $\mathbb{T}_{\mathbb{C}} \cong (\mathbb{C}^{*})^{n}$, the Poisson structure $C^{R}$ is given by the following expression
\[
C^{R} = \tfrac{1}{2}\sum_{i,j} C_{ij} w_{i} \partial_{w_{i}} \wedge w_{j} \partial_{w_{j}}.
\]
The Poisson structure $C^{R} \times 0$ is invariant under the $\mathbb{T}_{\mathbb{C}} \times \mathbb{T}_{\mathbb{C}}$ action on $\mathbb{T}_{\mathbb{C}} \times M$. Hence, it is also invariant under the anti-diagonal $\mathbb{T}_{\mathbb{C}}$ action given by $u \ast (w,m) = (wu^{-1}, um)$. As a result, the quotient $(\mathbb{T}_{\mathbb{C}} \times M)/\mathbb{T}_{\mathbb{C}}$ inherits a Poisson structure, and it is straightforward to check that it is isomorphic to $(M, \sigma_{C})$. Furthermore, this Poisson structure is invariant under the residual $\mathbb{T}_{\mathbb{C}}$ action. 

Xu's method now proceeds by constructing the symplectic groupoid of $\sigma_{C}$ via symplectic reduction, starting from the symplectic groupoid of $C^{R} \times 0$. More precisely, the symplectic groupoid of $C^{R} \times 0$ is given by the product of the symplectic groupoids $\mathcal{G}_{C}$ and $T^{*}M$ of $C^{R}$ and $0$, respectively. The Poisson action of $\mathbb{T}_{\mathbb{C}} \times \mathbb{T}_{\mathbb{C}}$ then lifts to a Hamiltonian action on $\mathcal{G}_{C} \times T^{*}M$, with moment map a groupoid homomorphism to $\mathfrak{t}^{*}_{\mathbb{C}} \times \mathfrak{t}^{*}_{\mathbb{C}}$. Taking the symplectic reduction with respect to the anti-diagonal action then yields the symplectic groupoid of $\sigma_{C}$, equipped with a residual Hamiltonian $\mathbb{T}_{\mathbb{C}}$ action:
\[
\mathcal{G}(\sigma_{C}) = (\mathcal{G}_{C} \times T^{*}M)/\!/_{0} \mathbb{T}_{\mathbb{C}}.
\]

In order to explicitly construct the symplectic groupoid of $\sigma_{C}$, we therefore need the symplectic groupoids of $(\mathbb{T}_{\mathbb{C}}, C^{R})$ and $(M, 0)$. The groupoid of $(M, 0)$ is the holomorphic cotangent bundle $T^{*}M$, with Hamiltonian $\mathbb{T}_{\mathbb{C}} $-action constructed in Section \ref{toricsection}. The symplectic groupoid of $(\mathbb{T}_{\mathbb{C}}, C^{R})$ may be constructed via the spray construction \cite{MR2900786,cabrera2018local}. It has been described explicitly in the work of Weinstein and Xu. 
\begin{theorem}\cite{xu1993poisson, weinstein1991symplectic} \label{WeinsteinXutorusgroupoid}
The holomorphic symplectic groupoid $\mathcal{G}_{C}$ integrating the holomorphic Poisson manifold $(\mathbb{T}_{\mathbb{C}}, C^{R})$ is given by the cotangent bundle $(T^{*}(\mathbb{T}_{\mathbb{C}}), \Omega_{0})$, where $\Omega_{0}$ is defined so that its imaginary part is the canonical symplectic form. Using the trivialization $T^{*}(\mathbb{T}_{\mathbb{C}}) \cong \mathfrak{t}_{\mathbb{C}}^{*} \times \mathbb{T}_{\mathbb{C}}$, the target and source maps are given, respectively, by 
\[
t(\alpha, w) = e^{\frac{i}{2}C(\alpha)} w, \qquad s(\alpha, w) = e^{-\frac{i}{2}C(\alpha)} w,
\]
and the multiplication is given by 
\begin{equation*}
m((\beta, e^{\frac{i}{2}C(\beta)}w), (\alpha, e^{-\frac{i}{2}C(\alpha)} w)) = (\beta + \alpha, e^{\frac{i}{2}C(\beta - \alpha)} w).
\end{equation*}
The multiplicative action of $\mathbb{T}_{\mathbb{C}}$ on the second factor of $\mathfrak{t}_{\mathbb{C}}^{*} \times \mathbb{T}_{\mathbb{C}}$ defines a holomorphic Hamiltonian action by groupoid isomorphisms, with moment map given by $J_{0}(\alpha, w) = i \alpha$, which is a groupoid homomorphism.

Finally, there is a (trivially $\mathbb{T}_{\mathbb{C}}$-equivariant) multiplicative prequantization given by 
\[(\mathcal{O}_{T^{*}(\mathbb{T}_{\mathbb{C}})}, d - 2 \pi i \alpha_{0}, \theta_{0}),\]
where $\alpha_{0}$ is the canonical primitive of $\Omega_{0}$, and $\theta_{0}$ is the multiplicative cocycle defined on the space of composable pairs of arrows by the formula 
\[
\theta_{0}((\beta, w_{1}), (\alpha, w_{2})) = e^{-i \pi C(\beta, \alpha)}. 
\]
\end{theorem}

Carrying through the symplectic reduction, using the explicit forms for the groupoids $\mathcal{G}_{C}$ and $T^{*}M$, we find that the symplectic groupoid of the Poisson structure $\sigma_{C}$ is given by the holomorphic cotangent bundle $(T^{*}M, \Omega_{0})$, but with deformed structure maps. More precisely, we find that the target and source maps are given, respectively, by 
\[
t(z) = e^{\frac{1}{2}CJ_{0}(z)} \pi(z), \qquad s(z) = e^{-\frac{1}{2}CJ_{0}(z)} \pi(z), 
\]
and the deformed multiplication is given by 
\begin{equation} \label{gpdproduct}
m(x,y) = e^{- \frac{1}{2}CJ_{0}(y)} x + e^{\frac{1}{2}CJ_{0}(x)}  y.
\end{equation}
In these expressions, $J_{0}$ is the holomorphic moment map for the $\mathbb{T}_{\mathbb{C}}$-action on $T^{*}M$, which was constructed in Section \ref{toricsection}, and which is a groupoid homomorphism to the abelian group $\mathfrak{t}^{*}_{\mathbb{C}}$. 

Xu's method may be upgraded to include the data of prequantizations. In the present case, we find that the symplectic groupoid of $\sigma_{C}$ carries a $\mathbb{T}_{\mathbb{C}}$-equivariant multiplicative prequantization. The line bundle with connection is given by $(\mathcal{O}_{T^{*}M}, d - 2\pi i \alpha_{0})$, where $\alpha_{0}$ is the canonical primitive for $\Omega_{0}$, and the multiplicative cocycle is given on the space of composable pairs of arrows by 
\begin{equation} \label{multcocyclelevel0}
\theta_{C}(x,y) = e^{i \pi C(J_{0}(x), J_{0}(y))}.
\end{equation}

\subsubsection{Monoidal category of A-branes}
We consider the A-model of the symplectic groupoid $(T^{*}M, \Im(\Omega_{0}))$. As a real symplectic manifold, this is just the usual cotangent bundle of $M$, and does not depend on the R-matrix $C$. As a result, its A-model is likewise independent of $C$. As we have seen above, the effect of turning on the R-matrix is to deform the groupoid structure maps (i.e. source, target and multiplication). \emph{This is reflected as a deformation of the monoidal structure on the category of A-branes.} The $0$-section $\mathcal{L}$ remains the monoidal unit. The product of branes is constructed as in Section \ref{Monoidalcotangentcat}, but with the groupoid product given by Equation \ref{gpdproduct}. 

\subsubsection{Lifting branes to the groupoid}
We now explain how the branes constructed in Section \ref{branessection} may be pulled-back to A-branes in the symplectic groupoid, following the proposal outlined in Section \ref{proposal}. Consider the product $M \times M$, equipped with the GC structure $\mathbb{J}_{C} \times \mathbb{J}_{C}^{\top}$, where $ \mathbb{J}_{C}^{\top}$ is the generalized complex structure corresponding to $-\sigma_{C}$. This may be viewed as a GC groupoid. The anchor map of the symplectic groupoid is obtained by combining the source and target maps
\[
(t, s) : (T^{*}M, Im( \Omega_{0})) \to (M \times M, \mathbb{J}_{C} \times \mathbb{J}_{C}^{\top}). 
\]
The graph of this map, which is identified with $T^*M$, is equipped with the trivial Hermitian line bundle with connection 
\[ d - 2\pi i Re(\alpha_{0}),\] 
which prequantizes $\Re(\Omega_{0})$. This defines a generalized complex brane in the product 
\[ (T^{*}M, -Im( \Omega_{0})) \times (M \times M, \mathbb{J}_{C} \times \mathbb{J}_{C}^{\top}).\] 
Such a brane is otherwise known as a generalized holomorphic map. Note that the multiplicative structure $ \exp(-i \pi Re(C(\beta, \alpha)))$ on the line bundle endows the anchor with the structure of a GC groupoid homomorphism. 

An important feature of such holomorphic maps is that they can be used to pullback branes. In this case, we may take GC branes on $M \times M$, and pull them back to A-branes in $T^*M$. We now work this out for the branes constructed in Section \ref{branessection}. Given a pair of natural numbers $(n, m) \in \mathbb{N}$, we have the following brane in $M \times M$
\[
B_{n} \boxtimes B_{m}^{*} = (L^{\otimes n}, h^{n}, \overline{ \nabla^{\otimes n}}) \boxtimes (L^{\otimes m}, h^{m}, \overline{ \nabla^{\otimes m}})^{*}.
\]
Pulling this back by the anchor map, we obtain the Hermitian line bundle with unitary connection
\begin{equation}\label{liftedunitarybranes}
\mathcal{B}_{n,m} \coloneqq t^{*}(L^{\otimes n}, h^{n}, \overline{ \nabla^{\otimes n}}) \otimes s^{*}(L^{\otimes m}, h^{m}, \overline{ \nabla^{\otimes m}})^{*} \otimes (\mathcal{O}_{T^{*}M}, 1, d - 2\pi i Re(\alpha_{0})),
\end{equation}
which is a space-filling A-brane in $T^{*}M$. As in Section \ref{holpq}, this space-filling brane has a holomorphic prequantization, obtained by adding the canonical primitive $2 \pi \Im( \alpha_{0})$ to all branes. This results in the following complex line bundles with connection 
\begin{equation}\label{liftedholpreqbranes}
(U_{n,m}, D_{n,m}) := (t^{*}(L^{\otimes n})\otimes s^{*}(L^{\otimes m})^{*}, t^{*}(\overline{ \nabla^{\otimes n}}) \otimes s^{*}(\overline{ \nabla^{\otimes m}})^{*} - 2\pi i \alpha_{0}),
\end{equation}
whose curvature is 
\[
\Omega_{n,m} = \Omega_{0} + t^{*}F_{n} - s^{*}F_{m}. 
\]

By results of \cite{bischoff2018morita, bischoffthesis}, $\Omega_{n,m}$ defines a holomorphic symplectic structure on $T^{*}M$ for a deformed complex structure, and the $(0,1)$-component of $D_{n,m}$ defines a complex structure on $U_{n,m}$, such that $D_{n,m}$ defines a holomorphic prequantization of $\Omega_{n,m}$. Furthermore, the space $T^{*}M$, equipped with these structures, defines a holomorphic symplectic Morita equivalence between the holomorphic Poisson structures $\sigma_{m}=\varphi_m^*\sigma_C$ and $\sigma_{n}=\varphi_n^*\sigma_C$. This means that the fibres of $t$ and $s$ are orthogonal with respect to $\Omega_{n,m}$, that $t$ defines a holomorphic Poisson map from $\Omega_{n,m}^{-1}$ to $\sigma_{n}$, and that $s$ defines a holomorphic Poisson map from $\Omega_{n,m}^{-1}$ to $-\sigma_{m}$. Let us denote this Morita equivalence by $Z_{n,m}$, and let $\mathcal{L}_{n,m}$ denote the zero section when viewed as a submanifold of $Z_{n,m}$. 

The branes on $M \times M$ satisfy $(B_{n} \boxtimes B_{m}^{*}) \ast (B_{m} \boxtimes B_{k}^{*}) \cong B_{n} \boxtimes B_{k}^{*}$, and hence the multiplicativity of the anchor map induces an isomorphism 
\begin{equation} \label{braneMultproperty}
\mathcal{B}_{n,m} \ast \mathcal{B}_{m, k} \cong \mathcal{B}_{n,k}. 
\end{equation}
The monoidal product of the branes $\mathcal{B}_{n,m}$ and $\mathcal{B}_{m, k}$ corresponds to the composition of the Morita equivalences $Z_{n,m}$ and $Z_{m, k}$. Therefore, the multiplicativity of the branes corresponds to a holomorphic product map 
\[
Z_{n,m} \times_{M} Z_{m,k} \to Z_{n, k},
\]
such that $\mathcal{L}_{n,m} \ast \mathcal{L}_{m,k} = \mathcal{L}_{n,k}$, and such that the multiplicative cocycle of Equation \ref{multcocyclelevel0} defines a flat morphism  
\[
\Theta_{C} : (U_{n, m}, D_{n,m}) \boxtimes (U_{m, k}, D_{m, k})|_{Z_{n,m} \times_{M} Z_{m,k}}  \to (U_{n,k}, D_{n,k}). 
\]

\subsubsection{Lifting the autoequivalence}
We explain how to lift the auto-equivalence $T_{C}$ to the A-model of the symplectic groupoid. Part of the data underlying $T_{C}$ is the Poisson automorphism $\varphi_{1}$ of $(M, Q)$. This integrates to a symplectic groupoid automorphism $\Phi$ of $(T^{*}M, \Im(\Omega_{0}))$. Hence, this acts on the A-model as a monoidal self-equivalence by pulling back branes. In particular, since $\Phi$ preserves the identity bisection, it satisfies $\Phi^{*}(\mathcal{L}) = \mathcal{L}$. We now verify that the groupoid automorphism acts as a shift operator on the collection of space-filling branes $B_{n,m}$ and thus on their holomorphic prequantizations:

\begin{lemma} \label{braneautolift}
There are canonical isomorphisms as follows for all $n,m$:
\[
\Phi^*(\mathcal{B}_{n,m}) \cong \mathcal{B}_{n+1, m+1}.
\]
Furthermore, this gives rise to analogous isomorphisms between the holomorphic prequantizations. 
\end{lemma}
\begin{proof}
The map $\varphi_{1}$ defines a holomorphic Poisson isomorphism from $\sigma_{1}$ to $\sigma_{0}$, and hence $\Phi$ defines an isomorphism between their holomorphic symplectic groupoids. The symplectic groupoid of $\sigma_{0}$ is the one described above. By Proposition 6.3 of \cite{bailey2016integration}, the symplectic groupoid of $\sigma_{1}$ is obtained from the symplectic groupoid of $\sigma_{0}$ by deforming the symplectic form to 
\[
\Omega_{1,1} = \Omega_{0} + t^{*}F_{1} - s^{*}F_{1}.
\]
Hence $\Phi^{*}(\Omega_{0}) = \Omega_{0} + t^{*}F_{1} - s^{*}F_{1}$. The holomorphic prequantization of $\Omega_{1,1}$ is given by $(U_{1,1}, D_{1,1})$. This line bundle has the same curvature as $\Phi^{*}(\mathcal{O}_{T^{*}(\mathbb{T}_{\mathbb{C}})}, d - 2 \pi i \alpha_{0})$, and the two bundles are canonically identified over $\mathcal{L}$. Using the connections, we therefore get a flat isomorphism between the bundles. Note that this isomorphism respects the multiplicative structures, as can be seen by comparing them over $\mathcal{L}$. Using the unitary version $(\mathcal{O}_{T^{*}(\mathbb{T}_{\mathbb{C}})}, d - 2 \pi i \Re(\alpha_{0}))$, we similarly obtain an isomorphism $\Phi^{*}(\mathcal{B}_{0,0}) \cong \mathcal{B}_{1,1}$. 

Now given a brane $\mathcal{B}_{n,m}$ we get 
\begin{align*}
\Phi^{*}(U_{n,m}) &= \Phi^* t^{*}(L^{\otimes n},  \overline{ \nabla^{\otimes n}}) \otimes \Phi^* s^{*}(L^{\otimes m}, \overline{ \nabla^{\otimes m}})^{*} \otimes \Phi^* (\mathcal{O}_{T^{*}M}, d - 2\pi i \alpha_{0}) \\
&\cong t^{*} \varphi_{1}^* (L^{\otimes n}, \overline{ \nabla^{\otimes n}}) \otimes s^{*} \varphi_{1}^{*}(L^{\otimes m},  \overline{ \nabla^{\otimes m}})^{*} \otimes U_{1,1} \\
&= t^{*}(T_{C}(B_{n})) \otimes s^* (T_{C}(B_{m}))^* \otimes (\mathcal{O}_{T^{*}M}, d - 2\pi i \alpha_{0}) \\
&= U_{n+1, m+1},
\end{align*}
and similarly for the branes. 
\end{proof}

\subsubsection{Algebras from twisted graded monoids} \label{algfromtwistedgrmon}
In order to construct a graded algebra from the sequence of generalized complex branes $B_{n}$, we follow the proposal of Section \ref{proposal}. The algebra is defined on the graded vector space 
\[
\mathcal{A}_{C} = \bigoplus_{n \geq 0} \Hom{\mathcal{L}, \mathcal{B}_{n,0}},
\]
with a product defined using the monoidal product and the autoequivalence $\Phi^{*}$. Analogously to the branes considered in Section \ref{Algfrommonoid}, the collection of branes $\{ \mathcal{B}_{n,0} \}_{n \in \mathbb{Z}}$ form a $\Phi$-twisted $\mathbb{Z}$-graded monoid. This means that instead of a map as in Equation \ref{braneproduct}, we have the following isomorphisms
\begin{equation} \label{twistedgradedmonoid}
(\Phi^{*})^{m}(\mathcal{B}_{n,0}) \ast \mathcal{B}_{m,0} \cong \mathcal{B}_{n+m, 0},
\end{equation}
which are the result of the morphisms of Equation \ref{braneMultproperty} and of Lemma \ref{braneautolift}. This allows us to define the product on the algebra as the following composition 
\begin{align*}
 &\Hom{\mathcal{L}, \mathcal{B}_{n,0}} \otimes  \Hom{\mathcal{L}, \mathcal{B}_{m,0}} \to   \Hom{\mathcal{L},(\Phi^{*})^{m}(\mathcal{B}_{n,0})} \otimes  \Hom{\mathcal{L}, \mathcal{B}_{m,0}} \to \\  &\Hom{\mathcal{L} \ast \mathcal{L},(\Phi^{*})^{m}(\mathcal{B}_{n,0}) \ast \mathcal{B}_{m,0}} \to  \Hom{\mathcal{L}, \mathcal{B}_{n + m,0}}.
\end{align*}

Just as in Section \ref{Congradedhomcoordring}, we may encode the $\Phi$-twisted $\mathbb{Z}$-graded monoid $\{ \mathcal{B}_{n,0} \}_{n \in \mathbb{Z}}$ as a holomorphic symplectic groupoid with prequantizing multiplicative line bundle and multiplicative system of brane bisections. Indeed, we have already seen that the branes $\mathcal{B}_{n,m}$ are encoded as prequantized Morita equivalences $Z_{n,m}$ relating holomorphic Poisson structures $\sigma_{m}$ and $\sigma_{n}$, and that the tensor product of branes corresponds to the composition of Morita equivalences. 

We now define $(Z_{n}, \Omega_{n})$ to be the Morita equivalence $(Z_{n,0}, \Omega_{n, 0})$, but where the map $t: Z_{n,0} \to M$ has been modified to $\varphi_{1} \circ t$, which is a holomorphic Poisson morphism from $(Z_{n}, \Omega_{n}^{-1})$ to $(M, I_{0}, \sigma_{0})$. Therefore, $(Z_{n}, \Omega_{n})$ defines a Morita self-equivalence of $\sigma_{C}$. In order to compose the Morita equivalences $Z_{n}$ and $Z_{m}$, we must now use the morphism $\Phi^{m}$ to pull-back the points of $Z_{n}$. In this way, we obtain a composition map 
\[
Z_{n} \times_{M} Z_{m} \to Z_{n + m}, \qquad (g, h) \mapsto m(\Phi^{-m}(g), h). 
\]
Assembling all of these maps together, the disjoint union $(\mathcal{Z}, \Omega) = \sqcup_{n \in \mathbb{Z}} (Z_{n}, \Omega_{n})$ has the structure of a holomorphic symplectic groupoid integrating the toric Poisson structure $\sigma_{C}$. In a similar way, we may equip the disjoint union $(\mathcal{U}, D) = \sqcup_{n \in \mathbb{Z}} (U_{n,0}, D_{n,0})$ with a multiplicative structure, making it into a multiplicative prequantization of $(\mathcal{Z}, \Omega)$. Furthermore, the submanifolds $\mathcal{L}_{n} := \mathcal{L}_{n,0} \subset Z_{n}$ are sections of the source and target maps which satisfy $\mathcal{L}_{n} \ast \mathcal{L}_{m} = \mathcal{L}_{n+m}$. 

\subsection{Prequantization via holomorphic symplectic reduction}
In this section, we give a construction via symplectic reduction of the prequantized symplectic groupoid with multiplicative bisections $(\mathcal{Z}, \Omega, \mathcal{U}, D, \{\mathcal{L}_{n} \}_{n \in \mathbb{Z}})$ considered in the previous section. This construction has the advantage of making transparent several additional structures which will be useful in identifying the graded algebra $\mathcal{A}_{C}$. Furthermore, as we will see, this groupoid may be viewed as a deformation of the groupoid from Section \ref{Congradedhomcoordring} which was used in the construction of the homogeneous coordinate ring. 

\begin{theorem} \label{deformedbrane}
Let $(\mathcal{Z}, \Omega)$ denote the holomorphic symplectic groupoid considered in Proposition \ref{multstruct1}, which consists of a disjoint union of holomorphic symplectic affine bundles for the cotangent bundle of $M$. This groupoid admits a deformation to a holomorphic symplectic groupoid integrating the toric Poisson structure $\sigma_{C}$. This deformation is given by fixing the underlying holomorphic symplectic manifold and deforming the structure maps. Namely, the deformed target and source maps are given, respectively, by 
\[
t(z) = e^{\frac{1}{2}CJ(z)} \pi(z), \qquad s(z) = e^{-\frac{1}{2}CJ(z)} \pi(z), 
\]
and the deformed multiplication is given by 
\[
m(x,y) = e^{- \frac{1}{2}CJ(y)} x + e^{\frac{1}{2}CJ(x)} y.
\]
The Hamiltonian action of $\mathbb{T}_{\mathbb{C}}$ continues to act by groupoid isomorphisms, and the moment map $J$ remains a groupoid homomorphism. 

The equivariant prequantization of Proposition \ref{multstruct1} also admits a deformation to a multiplicative prequantization $(\mathcal{U}, D, \Theta_{C})$. The underlying line bundle with connection $(\mathcal{U}, D)$ is as before, but the multiplicative cocycle is deformed. The deformed cocycle is defined on the space of composable arrows to be 
\begin{equation} \label{BigRmatmultcocycle}
\Theta_{C}(x,y)(a \otimes b) = e^{i \pi C(J(x), J(y))}  \Theta(e^{- \frac{1}{2}CJ(y)}  a, e^{\frac{1}{2}CJ(x)}  b),
\end{equation}
where $a \in \mathcal{U}_{x}$ and $b \in \mathcal{U}_{y}$, and $\Theta$ is the multiplicative structure of Proposition \ref{multstruct1}.
Furthermore, there is a canonical holomorphic $\mathbb{T}_{\mathbb{C}}$-equivariant isomorphism
\begin{equation} \label{equivpullbackiso}
\phi_{n} : U_{n} \stackrel{\cong}{\longrightarrow}  s^{*}(L^{\otimes n}),
\end{equation}
where $U_{n}$ is the restriction of $\mathcal{U}$ to $Z_{n}$. This isomorphism is obtained by combining the isomorphism $U_{n} \cong \pi^{*}(L^{\otimes n})$ from Proposition \ref{reductionprop} with the following map which is induced by the equivariant structure 
\[
\exp(-\tfrac{1}{2}CJ): \pi^{*}(L^{\otimes n}) \to s^{*}(L^{\otimes n}).
\]

Finally, there is a family of Lagrangian branes $\mathcal{L}_{n} \subset Z_{n}$ which satisfy $\mathcal{L}_{n} \ast \mathcal{L}_{m} = \mathcal{L}_{n+m}$ for all $(n,m) \in \mathbb{Z}^{2}$. These branes are sections for both $s$ and $t$, and the brane $\mathcal{L}_{n}$ induces the diffeomorphism $\varphi_{n}$ on $M$. Viewing $\mathcal{L}_{n}$ as a section of $s$, the restrictions of $\Omega_{n}$ and $J_{n}$ to $\mathcal{L}_{n}$ are given by 
\[
\mathcal{L}_{n}^{*}\Omega_{n} = F_{n}, \qquad \mathcal{L}_{n}^{*}J_{n} = -i \int_{0}^{n}\varphi_{t}^{*}(\mu) dt. 
\] 
Furthermore, the restriction of the isomorphism of Equation \ref{equivpullbackiso} (using the source map) to $\mathcal{L}_{n}$ yields the isomorphism 
\[
\mathcal{L}_{n}^{*}(\phi_{n}): \mathcal{L}_{n}^{*}(U_{n}, D_{n}) \stackrel{\cong}{\longrightarrow} e^{-R_{n}} (L^{\otimes n}, \overline{ \nabla^{\otimes n}}) e^{R_{n}},
\]
where $e^{R_{n}}$ is a gauge transformation, and 
\[
R_{n} = i \pi \int_{0}^{n} \int_{0}^{t} C(\varphi_{t}^{*}\mu, \varphi_{s}^{*}\mu) ds dt.
\]
\end{theorem}
\begin{proof}
The starting point is the holomorphic Poisson structure $\hat{\sigma}_{C}$ on $P_{L}$ which is defined by the equivariant structure on the line bundle, and which defines the Poisson module structure. This Poisson structure is invariant under the $\mathbb{T}_{\mathbb{C}}$-action defined by the equivariant structure, and by the $\mathbb{C}^{*}$-scaling action on the fibres. These lift to Hamiltonian actions on the holomorphic cotangent bundle $T^*P_{L}$. Recall from Section \ref{reductionsection} that the moment map for the scaling action is given by $i\mathcal{E}$; we denote the moment map for the $\mathbb{T}_{\mathbb{C}}$-action by $\tilde{J}$. 

By using the method of Xu outlined in Section \ref{XuConstruction} we construct the symplectic groupoid of $\hat{\sigma}_{C}$. It is given by the holomorphic cotangent bundle $(T^{*}P_{L}, \Omega_{0})$, but with deformed structure maps. Namely, the target and source maps are given, respectively, by 
\[
t(p) = e^{\frac{1}{2}C\tilde{J}(p)} \pi(p), \qquad s(p) = e^{-\frac{1}{2}C\tilde{J}(p)} \pi(p), 
\]
and the deformed multiplication is given by 
\[
m(p,q) = e^{- \frac{1}{2}C\tilde{J}(q)} p + e^{\frac{1}{2}C\tilde{J}(p)}  q.
\]
With respect to this new groupoid structure, the action of $\mathbb{T}_{\mathbb{C}} \times \mathbb{C}^{*}$ is by groupoid automorphisms, and the moment maps $\tilde{J}$ and $i \mathcal{E}$ are groupoid homomorphisms. 

This method extends to the prequantization. Namely, the trivial bundle with connection $d - 2 \pi i \alpha_{0}$, and trivial equivariant structure, defines an equivariant prequantization of $\Omega_{0}$. It is furthermore equipped with a multiplicative structure, given by the following cocycle on the space of composable pairs of arrows:
\begin{equation} \label{multstructurebeforered}
\tilde{\Theta}(p, q) = e^{i \pi C(\tilde{J}(p), \tilde{J}(q))}.
\end{equation}

Now as in the proof of Proposition \ref{multstruct1} we may use the $\mathbb{C}^{*}$-action to reduce $T^{*}P_{L}$ at a discrete subgroup of $\mathbb{C}$. In this way we obtain
\[
\mathcal{Z} = \mathcal{E}^{-1}(\tfrac{1}{2 \pi} \mathbb{Z})/\mathbb{C}^{*},
\]
which is a holomorphic symplectic groupoid over $M$ integrating the holomorphic Poisson structure $\sigma_{C}$. 

Because the holomorphic symplectic manifold $(T^{*}P_{L}, \Omega_{0})$, the group actions, the moment maps, and the equivariant prequantization are independent of the R-matrix $C$, the reduced structures are as in Proposition \ref{multstruct1}. Hence the resulting holomorphic symplectic manifold $(\mathcal{Z}, \Omega)$ is a disjoint union of the holomorphic affine bundles constructed in Proposition \ref{reductionprop}, and they are equipped with the same Hamiltonian $\mathbb{T}_{\mathbb{C}}$-action as constructed in Proposition \ref{Hamiltonianlift}. Furthermore, the prequantization of $T^{*}P_{L}$ descends to the same equivariant prequantizations $(U_{n}, D_{n})$ as constructed in Proposition \ref{liftedequivariantaction}. 

On the other hand, the structure maps of the groupoid are deformed. By construction, it follows that with respect to these new structures the action of $\mathbb{T}_{\mathbb{C}}$ is by groupoid automorphisms and its moment map $J$ is a groupoid homormorphism to the abelian group $\mathfrak{t}_{\mathbb{C}}^*$. A straightforward reduction of the given structure maps of $T^*P_L$ yields the deformed structure maps in the statement of the theorem. We explain here how to obtain the expression for the multiplicative cocycle on the line bundle.

On $T^{*}P_{L}$, the product of two elements $a \in \mathcal{O}|_{p}$ and $b \in  \mathcal{O}|_{q}$ is obtained by first transporting $a$ and $b$ to $p'=\exp(-\tfrac{1}{2}C\tilde J(p))p$ and $q'= \exp(\tfrac{1}{2}C\tilde J(q))q$, respectively (using the trivial equivariant structure, which does not affect the values of $a,b$), which places them in the same cotangent fibre, then taking their usual product, placed above the sum $p'+q'$ of cotangent vectors, and finally multiplying the result by the value of the cocycle~\eqref{multstructurebeforered}.   This operation descends to a similar operation on $\mathcal{U}$, but with nontrivial transport maps.  The usual product of $a,b$ descends to the tensor product, i.e. the original multiplicative structure $\Theta$ of Proposition \ref{multstruct1}. Therefore, we arrive at the expression in Equation \ref{BigRmatmultcocycle}.

The isomorphism of Equation \ref{equivpullbackiso} is constructed as in the proof of Proposition \ref{reductionprop}. Namely, there is a canonical section $s_{n}$ of the pullback of $L^{\otimes n}$ to $P_{L}$, and if we pull it back via the source map to $T^{*}P_{L}$, then we see that it reduces to $Z_{n}$, providing the desired isomorphism. Tracing through this construction, we see that this isomorphism differs from the one constructed in Proposition \ref{reductionprop} by the equivariant action of $\exp(- \tfrac{1}{2}CJ)$. 

In order to construct the submanifolds $\mathcal{L}_{n} \subset Z_{n}$, we start with the metric $h$ which is defined on $P_{L}$. The unit $S^{1}$-bundle $S$ defined by the equation $ h = 1$ is a coisotropic submanifold of $(P_{L}, \hat{Q})$. Hence, it integrates to a Lagrangian subgroupoid $C_{S}$ of $(T^{*}P_{L}, \Im(\Omega_{0}))$ (generalising the conormal bundle from Section \ref{reductionsection}). The submanifolds $\mathcal{L}_{n}$ are defined by reducing this subgroupoid:
\[
\mathcal{L}_{n} = (C_{S} \cap \mathcal{E}^{-1}(\tfrac{n}{2\pi}))/S^1. 
\]
In order to deduce the properties of $\mathcal{L}_{n}$, we use a different construction which is more convenient. Let 
\[
W_{h} \coloneqq \tfrac{1}{4 \pi} (\Im \Omega_{0})^{-1}(t^{*} d \log(h))
\]
be the Hamiltonian vector field of $\frac{1}{4 \pi} t^{*}\log(h)$. This vector field is tangent to the source fibres, is $t$-related to $\hat{W}$ (Equation~\ref{hatw}), is $\mathbb{C}^{*}$-invariant (since $d\log(h)$ is), and satisfies 
\[
W_{h}(\mathcal{E}) = \tfrac{1}{2 \pi},
\]
since $h$ is $S^1$-invariant and has weight $2$ with respect to the radial rescaling of the fibres of $P_{L}$. 
Let $\chi_{t}$ be the flow of $W_{h}$, let $\epsilon : P_{L} \to T^{*}P_{L}$ be the zero section, and let 
\[
\mathcal{L}(t) = \chi_{t} \circ \epsilon : P_{L} \to T^{*}P_{L}. 
\]
Then $\mathcal{L}(t)$ is a $\mathbb{C}^{*}$-equivariant section of $s$, which is $t$-related to the time $t$-flow $\hat{\varphi}_{t}$ of $\hat{W}$, and which has its image contained in $\mathcal{E}^{-1}(\frac{t}{2\pi})$. Hence, taking the quotient by $\mathbb{C}^{*}$, we obtain a sequence of maps 
\[
\mathcal{L}_{n} = \mathcal{L}(n)/\mathbb{C}^{*} : M \to Z_{n},
\]
which satisfy $s \circ \mathcal{L}_{n} = id$ and $t \circ \mathcal{L}_{n} = \varphi_{n}$. In fact, the image of these maps coincides with the submanifolds $\mathcal{L}_{n}$ considered above (hence the abuse of notation). To see this, note that $C_{S}$ is a groupoid which is diffeomorphic to $S \times \mathbb{R}$, and a generating section of its Lie algebroid is given by $W_{h}|_{S}$. Hence, the flow of $W_{h}$ defines the embedding of $C_{S}$ into $T^{*}P_{L}$, from which it follows that $\mathcal{L}(t)(S) = C_{S} \cap \mathcal{E}^{-1}(\frac{t}{2\pi})$.

The properties of $\mathcal{L}_{n}$ now follow from those of $W_{h}$. First, taking the Lie derivative of the moment map $\tilde{J}$, we obtain 
\begin{align*}
\mathcal{L}_{W_{h}}(\tilde{J}) &= d\tilde{J}(W_{h}) = \Omega_{0}(W_{h}, V') \\
&= \frac{i}{2 \pi} \langle \partial (t^{*}\log(h)), V' \rangle \\
&= \frac{i}{2 \pi} t^{*} (V(\log(h)))\\
&= -i t^{*}\mu,
\end{align*}
where $V'$ is the infinitesimal action of $\mathbb{T}_{\mathbb{C}}$ on $T^{*}P_{L}$, and the last line follows from Equation \ref{mommappot} and the relationship between $h$ and the $\mathbb{T}$-invariant K\"{a}hler potentials. From this the equation for $\mathcal{L}_{n}^*J_{n}$ follows. 

In order to compare $\mathcal{L}_{n}^{*}(U_{n}, D_{n})$ and $(L^{\otimes n}, \overline{\nabla^{\otimes n}})$, we pull them both back to $P_{L}$. When we do this, the identification between the line bundles induced by the isomorphism of Equation \ref{equivpullbackiso}, reduces to the canonical isomorphism $s_{n}$ between the trivial bundle and the pullback of $L^{\otimes n}$. This is because, as explained above, the isomorphism of Equation \ref{equivpullbackiso} may be constructed from $s_{n}$. The upshot is that the comparison is now reduced to the comparison of two connections on the trivial bundle over $P_{L}$. 

First, if we pullback $(L^{\otimes n}, \overline{\nabla^{\otimes n}})$ and use $s_{n}$ to trivialize it we obtain 
\[
d + \int_{0}^{n} \hat{\varphi}_{t}^{*}(\partial \log(h)) dt. 
\]
Second, if we pullback $\mathcal{L}_{n}^{*}(U_{n}, D_{n})$, we obtain the trivial bundle with the connection 
\begin{equation} \label{middlecalccompofcon}
d - 2 \pi i \mathcal{L}(n)^{*}(\alpha_{0}).
\end{equation}
To compute this, we take the Lie derivative of the tautological form $\alpha_{0}$: 
\[
\mathcal{L}_{W_{h}}(\alpha_{0}) = \iota_{W_{h}} \Omega_{0} + d \langle \alpha_{0}, W_{h} \rangle = \tfrac{i}{2\pi} t^* \partial \log(h) + d \langle \alpha_{0}, W_{h} \rangle. 
\]
By definition of $\alpha_{0}$, we have $\langle \alpha_{0}, W_{h} \rangle(p)  = i \langle p, \pi_{*}(W_{h}) \rangle$. Then using the fact that $\pi = s \circ \exp(\frac{1}{2}C\tilde{J})$, that $W_{h}$ points along the source fibres, and the above expression for the Lie derivative of $\tilde{J}$, we find
\[
\pi_{*}(W_h) = -\tfrac{i}{2} V \circ C \circ t^{*}\mu,
\]
where $V$ is the infinitesimal generator of the $\mathbb{T}_{\mathbb{C}}$-action on $P_{L}$. Assembling these facts, we find that $\langle \alpha_{0}, W_{h} \rangle = -\frac{i}{2}C(t^{*}\mu, \tilde{J})$. Therefore, we find that Equation \ref{middlecalccompofcon} is given by 
\[
d + \int_{0}^{n} \hat{\varphi}_{t}^{*}(\partial \log(h)) dt + dR_{n}, 
\]
where 
\[
R_{n} = i \pi \int_{0}^{n} \int_{0}^{t} C(\varphi_{t}^{*}\mu, \varphi_{s}^{*}\mu) ds dt.
\]
Hence, it follows that 
\[
\mathcal{L}_{n}^{*}(U_{n}, D_{n}) = e^{-R_{n}} \overline{\nabla^{\otimes n}} e^{R_{n}}. 
\]
The formula for $\mathcal{L}_{n}^*\Omega_{n}$ then follows by taking the curvature of the connections. 
\end{proof}

The result of Theorem \ref{deformedbrane} is a uniform construction, using symplectic reduction, of a symplectic groupoid starting from the data of the Poisson module $(L, h, \hat{\sigma}_{C})$. One of the upshots is an independent construction of the branes $B_{n}$ and the autoequivalence $T_{C}$ constructed in Section \ref{branessection} (note however that the data of the hermitian metric has been lost). In Appendix \ref{Comparisonappendix} we will explicitly construct an isomorphism between the prequantized groupoid of Theorem \ref{deformedbrane} and the groupoid sketched in Section \ref{algfromtwistedgrmon}. 

\subsection{The noncommutative homogeneous coordinate ring}
In this section we will identify the non-commutative algebra defined in Section \ref{algfromtwistedgrmon}. The derivation of this algebra closely follows Sections \ref{thequantization} and \ref{algebraconst}, with minor changes required to incorporate the non-trivial R-matrix. 

First, we compute the space of homomorphisms $\Hom{\mathcal{L}, \mathcal{B}_{n,0}}$, using the holomorphic moment map $J_{n}: Z_{n} \to \mathfrak{t}_{\mathbb{C}}^{*}$. Let $\mathcal{F}_{BS}(\mathcal{L}_{n},   \mathcal{B}_{n,0})$ denote the set of Bohr-Sommerfeld fibres of $J_{n}$ which have a non-trivial intersection with $\mathcal{L}_{n}$. Recall that the space of homomorphisms is defined to be the direct sum over $\mathcal{F}_{BS}(\mathcal{L}_{n},   \mathcal{B}_{n,0})$ of the vector spaces of covariantly constant sections of $U_{n}$ restricted to the Bohr-Sommerfeld fibres (c.f. Equation \ref{proposalforquant}): 
\[
\Hom{\mathcal{L}_{n}, \mathcal{B}_{n,0}}\coloneqq \bigoplus_{J_{n}^{-1}(\xi) \in \mathcal{F}_{BS}(\mathcal{L}_{n}, \mathcal{B}_{n,0})} H^{0}(J_{n}^{-1}(\xi), U_{n})^{D_{n}}. 
\]
Since the prequantized holomorphic integrable system $(Z_{n}, \Omega_{n}, U_{n}, D_{n}, J_{n})$ is unchanged from Section \ref{thequantization}, we may apply Proposition \ref{BSdetermination} to determine the Bohr-Sommerfeld fibres. However, since the Lagrangian $\mathcal{L}_{n}$ has been deformed, we must use Theorem \ref{deformedbrane} to check that $\mathcal{F}_{BS}(\mathcal{L}_{n},   \mathcal{B}_{n,0})$ is unchanged. Once this is done, we may apply Theorem \ref{QuantizationTheorem} to identify the space of homomorphisms with $H^{0}(M, L^{\otimes n})$. 

\begin{proposition}
Let $(Z_{n}, \Omega_{n}, U_{n}, D_{n}, J_{n}, \mathcal{L}_{n})$ be a component of the groupoid constructed in Theorem \ref{deformedbrane}. Then the fibre $J_{n}^{-1}(\xi)$ has non-trivial intersection with the brane $\mathcal{L}_{n}$ if and only if  $\xi \in -i n \Delta$, where $\Delta  = \mu(M)$ is the image of the real moment map $\mu$ (The Delzant polytope). Therefore, the space $\mathcal{F}_{BS}(\mathcal{L}_{n},   \mathcal{B}_{n,0})$ of Bohr-Sommerfeld fibres intersecting $\mathcal{L}_{n}$ is unchanged from Proposition \ref{BSdetermination}, and as a result, there is a canonical isomorphism 
\[
\Hom{\mathcal{L}_{n}, \mathcal{B}_{n,0}} \cong H^{0}(M, L^{\otimes n}). 
\]
\end{proposition}
\begin{proof}
By Theorem \ref{deformedbrane}, we know that 
\[
\mathcal{L}_{n}^*J_{n} = -i \int_{0}^{n}\varphi_{t}^{*}(\mu) dt,
\]
where $\varphi_{t}$ is the flow of the vector field $W$. Hence, we need only check that the image of this map is $-in \Delta$. The polytope $\Delta \subset \mathfrak{t}^{*}$ is a compact convex simple polytope, which is defined by intersecting half-planes:
\[
\Delta = \{ \xi \in \mathfrak{t}^{*} \ | \ \langle v_{i} , \xi \rangle \geq \lambda_{i}, i = 1, ... , d \},
\]
where $v_{i} \in \mathfrak{t}$. The faces of $\Delta$ are given by various sets of equalities $\langle v_{i} , \xi \rangle = \lambda_{i}$, and they correspond in turn to the orbits of the $\mathbb{T}_{C}$-action. Since the vector field $W$ is constructed from this action, it therefore preserves the orbits and hence the various faces of $\Delta$. In other words, the flow $\varphi_{t}$ preserves both equalities and inequalities. Now given a point $x \in M$, we obtain 
\[
\langle v_{i} ,  \int_{0}^{n}\mu(\varphi_{t}(x)) dt \rangle = \int_{0}^{n} \langle v_{i} , \mu(\varphi_{t}(x))  \rangle dt \geq \int_{0}^{n} \lambda_{i} dx = n\lambda_{i}, 
\]
implying that the image of $\mathcal{L}_{n}^*J_{n}$ is contained in $ -i n \Delta$. Furthermore, each $\mathbb{T}_{\mathbb{C}}$-orbit must map into the corresponding face of $ -i n \Delta$. Now the vector $W$ commutes with the $\mathbb{T}$-action, implying that the map $\mathcal{L}_{n}^*J_{n}$ factors through the quotient $M/\mathbb{T} \cong \Delta$. The upshot is that we have a map $\Delta \to  -i n \Delta$ which preserves faces. Hence this map is surjective. 
\end{proof}

The construction of the product on the algebra now follows as in Section \ref{algebraconst}. Namely, consider two elements of the space of homomorphisms, i.e. flat sections $f_{n, \xi}$ and $f_{m, \eta}$ of $U_{n}$ and $U_{m}$, respectively, which are supported on Bohr-Sommerfeld fibres $J_n^{-1}(\xi)$ and $J_m^{-1}(\eta)$, respectively. We then use the multiplicative structure $\Theta_{C}$, which was constructed in Theorem \ref{deformedbrane}, to multiply them, yielding a flat section 
\[
\Theta_{C}(f_{n, \xi}, f_{m, \eta}) 
\]
of $U_{n+m}$, supported on the product of the Bohr-Sommerfeld fibres, which is $J_{n+m}^{-1}(\xi+\eta)$. We now identify this algebra with a deformation of the homogeneous coordinate ring.

\begin{theorem}
The algebra $\mathcal{A}_{C}$ and the homogeneous coordinate ring of the toric variety have canonically isomorphic underlying graded vector spaces 
\[
\mathcal{A}_{C} \cong \bigoplus_{n \geq 0} H^{0}(M, L^{\otimes n}). 
\]
Under this identification the product of two homogeneous sections $f$ and $g$, with respective $\TT_\CC$-weights $w_{1}, w_{2}\in \mathfrak{t}_\CC^*$, is given by 
\[
f \ast g = e^{\frac{i}{4 \pi} C(w_{1}, w_{2})} f \otimes g,
\]
resulting in the commutation relations
\[
f\ast g = e^{\frac{i}{2 \pi} C(w_{1}, w_{2})}g\ast f. 
\]
\end{theorem}
\begin{proof}
Let $\xi \in \mathfrak{t}^{*}_{\mathbb{C}}$ such that $J_{n}^{-1}(\xi) \in \mathcal{F}_{BS}(\mathcal{L}_{n}, \mathcal{B}_{n,0})$. As in Section \ref{thequantization}, let $s_{n,\xi}: M^{o} \to Z_{n}$ be the equivariant section of the projection $\pi$. This map satisfies 
\[
t\circ s_{n, \xi}(x) = e^{\frac{1}{2}C\xi}x, \qquad s\circ s_{n, \xi}(x) = e^{-\frac{1}{2}C\xi}x.
\]
The groupoid product of two of these maps satisfies 
\[
m(s_{\xi, n}(e^{\frac{1}{2}C\eta}x), s_{\eta, m}(e^{-\frac{1}{2}C\xi}x)) = s_{\xi + \eta, n + m}(x).
\]
Recall that $s^{*}_{n, \xi}(U_{n}) \cong L^{\otimes n}|_{M^{0}}$, and this defines an isomorphism from $H^{0}(J_{n}^{-1}(\xi), U_{n})^{D_{n}}$ to the weight $2\pi i \xi$ subspace of $H^{0}(M^{o}, L^{\otimes n})$. Now consider two morphisms 
\[
f_{n, \xi} \in \Hom{\mathcal{L}_{n}, \mathcal{B}_{n,0}}, \qquad f_{m , \eta} \in \Hom{\mathcal{L}_{m}, \mathcal{B}_{m,0}}. 
\]
Pulling these back via the respective isomorphisms defined by $s_{n, \xi}$ and $s_{m, \eta}$ allows us to view these as sections of $L^{\otimes n}$ and $L^{\otimes m}$. We suppress this identification from the notation. 

Their product is an element of $\Hom{\mathcal{L}_{n+m}, \mathcal{B}_{n+m,0}}$, which is given by 
\[
\Theta_{C}(f_{n, \xi}\boxtimes f_{m, \eta})(x) = e^{i \pi C(\xi, \eta)} \Theta( e^{-\frac{1}{2}C\eta} f_{n, \xi}(e^{\frac{1}{2}C\eta}x), e^{\frac{1}{2}C\xi} f_{m, \eta}(e^{-\frac{1}{2}C\xi}x)).
\]
Pulling this back via the isomorphism defined by $s_{\xi + \eta, n + m}$ allows us to view it as a section of $L^{\otimes (n + m)}$. Using the fact that $f_{n, \xi}$ and $f_{m, \eta}$ are equivariant sections of respective weights $2\pi i \xi$ and $2\pi i \eta$, we may simplify the above expression to 
\begin{align*}
\Theta_{C}(f_{n, \xi}\boxtimes f_{m, \eta}) &=  e^{i \pi C(\xi, \eta)}  (e^{\pi i C(\eta, \xi)}  f_{n, \xi} \otimes e^{-\pi i C(\xi, \eta)}f_{m, \eta}) \\
&= e^{-i \pi C(\xi, \eta)} f_{n, \xi}\otimes f_{m, \eta}. 
\end{align*}
\end{proof}

\appendix
\section{Comparison of groupoids} \label{Comparisonappendix}
In this appendix we compare the two constructions of prequantized (graded) holomorphic symplectic groupoids integrating the toric Poisson structure $\sigma_{C}$: the groupoid whose construction was sketched at the end of Section \ref{XuConstruction}, and the groupoid constructed in Theorem \ref{deformedbrane}. 

\subsection{Magnetic deformation construction} \label{appmagdef}
We first carefully flesh out the construction of the groupoid outlined in Section \ref{XuConstruction}. The starting data is the (prequantized) source simply connected holomorphic symplectic groupoid $\mathcal{G}_{C}$ integrating $\sigma_{C}$. Recall that as a holomorphic symplectic manifold, it is given by the holomorphic cotangent bundle $(T^{*}M, \Omega_{0})$, where $\Omega_{0}$ is the holomorphic symplectic form whose imaginary part is the real canonical symplectic form. The target and source maps are given, respectively, by 
\[
t(z) = e^{\frac{1}{2}CJ_{0}(z)} \pi(z), \qquad s(z) = e^{-\frac{1}{2}CJ_{0}(z)} \pi(z), 
\]
and the deformed multiplication is given by 
\begin{equation*}
m(x,y) = e^{- \frac{1}{2}CJ_{0}(y)} x + e^{\frac{1}{2}CJ_{0}(x)}  y,
\end{equation*}
where $J_{0} : T^{*}M \to \mathfrak{t}_{\mathbb{C}}^{*}$ is the holomorphic moment map generating the cotangent lift of the action of $\mathbb{T}_{\mathbb{C}}$. It is defined, for $\alpha \in T^{*}M$ and $u \in \mathfrak{t}_{\mathbb{C}}$ by 
\[
\langle J_{0}(\alpha), u \rangle = i \langle \alpha, V_{u} \rangle, 
\]
where $V_{u}$ is the toric vector field corresponding to $u$. 

This symplectic groupoid has a multiplicative prequantization given by the trivial line bundle with connection $d - 2 \pi i \alpha_{0}$, where $\alpha_{0}$ is the canonical primitive of $\Omega_{0}$, and mutliplicative cocycle defined on the space of composable pairs of arrows by 
\[
\theta_{C}(x,y) = e^{i \pi C(J_{0}(x), J_{0}(y))}.
\]

Now consider the autoequivalence $T_{C}$ and the sequence of branes 
\[
B_{n} \coloneqq T^{n}_{C}(B_{0}) = (L^{\otimes n}, \overline{ \nabla^{\otimes n}}),
\]
constructed in Section \ref{branessection}. The curvature of $ \overline{ \nabla^{\otimes n}}$ is given by $-2\pi i F_{n}$. Using this data, we will construct a graded prequantized holomorphic symplectic groupoid. The following is based on the constructions of \cite{bailey2016integration,bischoff2018morita, bischoffthesis}.

\begin{enumerate}
\item Consider the open cover of $M$ consisting of a $\mathbb{Z}$-indexed family of copies of $M$: $\{ M_{i} \}_{i \in \mathbb{Z}}$. We \emph{localize} the groupoid $\mathcal{G}_{C}$ to this cover. The result is a groupoid 
\[
\mathcal{G}_{C,\mathbb{Z}} = \bigsqcup_{(i,j) \in \mathbb{Z}^2} \mathcal{G}_{C, i,j},
\]
defined over the disjoint union $M_{\mathbb{Z}} = \sqcup_{i \in \mathbb{Z}} M_{i}$. The groupoid structures are defined as above, but the arrows in $\mathcal{G}_{C, i,j}$ go from $M_{j}$ to $M_{i}$. As a result, it is clear that $\mathcal{G}_{C,\mathbb{Z}}$ is a prequantized holomorphic symplectic groupoid integrating the holomorphic Poisson structure consisting of $\sigma_{C}$ on each copy $M_{i}$. 

\item Consider the real closed $2$-form $F_{\mathbb{Z}}$ on $M_{\mathbb{Z}}$ defined to be $F_{i}$ on $M_{i}$. Recall that this form satisfies Equation \ref{starequation}. Therefore, by \cite[Proposition 6.3]{bailey2016integration}, the deformed complex form $\Omega_{0} + t^{*}F_{\mathbb{Z}} - s^{*}F_{\mathbb{Z}}$ defines a new holomorphic symplectic structure on $\mathcal{G}_{C,\mathbb{Z}}$, such that it becomes a holomorphic symplectic groupoid over the manifold $M_{\mathbb{Z}}$, now equipped with the holomorphic Poisson structure $(I_{i}, \sigma_{i})$ on $M_{i}$. Note that the holomorphic symplectic form on $\mathcal{G}_{C, i,j}$ is given by 
\[
\Omega_{i,j} = \Omega_{0} + t^{*}F_{i} - s^{*}F_{j}. 
\]
Denote the new groupoid $\mathcal{Z}_{C, \mathbb{Z}} = \sqcup_{(i,j) \in \mathbb{Z}^2} Z_{i,j}$. 

\item The $2$-form $F_{\mathbb{Z}}$ is prequantized by the line bundle with connection $(A, \nabla_{A})$ which consists of $(L^{\otimes i}, \overline{ \nabla^{\otimes i}})$ on $M_{i}$. By \cite[Theorem 4.4.5]{bischoffthesis}, this leads to a holomorphic multiplicative prequantization of $\mathcal{Z}_{C, \mathbb{Z}}$ given by 
\[
(\mathcal{U}_{C, \mathbb{Z}}, D_{C, \mathbb{Z}}) = t^{*}(A, \nabla_{A}) \otimes s^{*}(A, \nabla_{A})^{*} \otimes (\mathcal{O}, d - 2\pi i \alpha_{0}). 
\] 
Indeed, the component over $Z_{i,j}$ is given by 
\[
(U_{i,j}, D_{i,j}) = (t^{*}(L^{\otimes i})\otimes s^{*}(L^{\otimes j})^{*}, t^{*}(\overline{ \nabla^{\otimes i}}) \otimes s^{*}(\overline{ \nabla^{\otimes j}})^{*} - 2\pi i \alpha_{0}),
\]
which is a complex line bundle with connection whose curvature equals $-2\pi i \Omega_{i,j}$. As a result, $D_{i,j}^{(0,1)}$ defines a holomorphic structure with respect to which $(U_{i,j}, D_{i,j})$ defines a holomorphic prequantization of $(Z_{i,j}, \Omega_{i,j})$. In order to construct a multiplicative structure for $(\mathcal{U}_{C, \mathbb{Z}}, D_{C, \mathbb{Z}})$, we need a flat $\delta$-closed trivialization of $\delta(\mathcal{U}_{C, \mathbb{Z}}, D_{C, \mathbb{Z}})$ over the space of composable pairs of arrows $\mathcal{Z}_{C, \mathbb{Z}}^{(2)}$. But we have the canonical identification 
\begin{align*}
\delta(\mathcal{U}_{C, \mathbb{Z}}, D_{C, \mathbb{Z}}) &= \delta(t^{*}(A, \nabla_{A}) \otimes s^{*}(A, \nabla_{A})^{*}) \otimes (\mathcal{O}, d - 2\pi i \delta(\alpha_{0})) \\
&= (\mathcal{O}, d - 2\pi i \delta(\alpha_{0})),
\end{align*}
so that a flat multiplicative prequantization is provided by $\theta_{C}$. We may view the complement of the zero section in $\mathcal{U}_{C, \mathbb{Z}}$ as a groupoid over $M_{\mathbb{Z}}$. It is a central extension of $\mathcal{Z}_{C, \mathbb{Z}}$. 

\item  Let $\mathcal{L} \subset \mathcal{G}_{C}$ be the identity bisection, and let $\mathcal{L}_{i,j}$ denote this bisection when it is viewed as a (non-holomorphic) submanifold of $Z_{i,j}$. Then $\mathcal{L}_{i,j}$ is a section of both the source and target, it defines the identification of $M_{i}$ and $M_{j}$ as smooth manifolds, and we have 
\[
\Omega_{i,j}|_{\mathcal{L}_{i,j}} = F_{i} - F_{j}.
\]
Furthermore, $\mathcal{L}_{i,j} \ast \mathcal{L}_{j,k} = \mathcal{L}_{i,k}$. 

\item Consider the $Q$-Poisson diffeomorphism $\varphi_{1}$ underlying $T_{C}$. By Equation \ref{newcomplexstructures}, $\varphi_{1}$ defines a holomorphic Poisson isomorphism from $(M, I_{i+1}, \sigma_{i+1})$ to $(M, I_{i}, \sigma_{i})$. We thus view it as a holomorphic Poisson automorphism of $M_{\mathbb{Z}}$. As explained in Section \ref{XuConstruction}, $\varphi_{1}$ integrates to a symplectic groupoid automorphism $\Phi$ of the groupoid $(T^{*}M, \Im(\Omega_{0}))$ integrating $Q$. We may view it as a (at the moment only $C^{\infty}$) groupoid automomorphism of $\mathcal{G}_{C, \mathbb{Z}}$ covering $\varphi_{1}$. It is defined by mapping $\mathcal{G}_{C, i,j}$ to $\mathcal{G}_{C, i-1,j-1}$. But now recall from Lemma \ref{braneautolift} that $\Phi^{*}(\Omega_{i,j}) = \Omega_{i+1, j+1}$. Hence, $\Phi$ defines a holomorphic symplectic groupoid automorphism of $\mathcal{Z}_{C, \mathbb{Z}}$. This map satisfies $\Phi(\mathcal{L}_{i,j}) = \mathcal{L}_{i-1,j-1}$. 

\item The automorphism $\Phi$ may be lifted to a flat multiplicative automorphism $\hat{\Phi}$ of $\mathcal{U}_{C, \mathbb{Z}}$. First, for every $(i,j) \in \mathbb{Z}^2$, we need a flat isomorphism  
\[
\hat{\Phi}_{i,j} : (U_{i+1, j+1}, D_{i+1, j+1}) \to \Phi^{*}(U_{i,j}, D_{i,j}),
\]
over $Z_{i+1, j+1}$. Since both connections have the same curvature we may first construct the isomorphism along $\mathcal{L}_{i+1, j+1}$, and then integrate it to all $Z_{i+1, j+1}$. Restricting the line bundles, we get 
\begin{align}
(U_{i+1, j+1}, D_{i+1, j+1})|_{\mathcal{L}_{i+1, j+1}} &= (L^{\otimes i+1} \otimes (L^{\otimes j+1})^{*}, \overline{ \nabla^{\otimes i+1}} \otimes (\overline{ \nabla^{\otimes j+1}})^{*}) \label{LB1} \\
 \Phi^{*}(U_{i,j}, D_{i,j})|_{\mathcal{L}_{i+1, j+1}} &= \varphi_{1}^{*}(L^{\otimes i} \otimes (L^{\otimes j})^{*}, \overline{ \nabla^{\otimes i}} \otimes (\overline{ \nabla^{\otimes j}})^{*}). \label{LB2}
\end{align}

Recall from Section \ref{branessection} that 
\[
(L^{\otimes n+1}, \overline{ \nabla^{\otimes n+1}}) \cong T_{C}(L^{\otimes n}, \overline{ \nabla^{\otimes n}}) = (L, \overline{ \nabla}) \otimes \varphi_{1}^{*}(L^{\otimes n}, \overline{ \nabla^{\otimes n}}).
\]
To be more precise, the isomorphism from the left hand side to the right hand side is given by $id \otimes \hat{\varphi}_{1}^{\otimes n}$, where $\hat{\varphi}_{1}$ is the time-$1$ flow of $\hat{W}$ on $L$. Therefore, we see that $(id \otimes \hat{\varphi}_{1}^{\otimes i}) \otimes (id \otimes \hat{\varphi}_{1}^{\otimes j})^{*}$ takes us from \ref{LB1} to \ref{LB2}. On the level of line bundles, this is the morphism 
\[
\hat{\varphi}_{1}^{\otimes(i-j)} : L^{\otimes i-j} \to \varphi_{1}^{*}(L^{\otimes i-j}). 
\]
Extending this to a flat isomorphism over $Z_{i+1, j+1}$ we obtain $\hat{\Phi}_{i,j}$. In order to check that the resulting map is a homomorphism, we need to establish that $\hat{\Phi}^{*}(\theta_{C}) = \theta_{C}$. Because the cocycle is flat, it suffices to check equality on $\mathcal{L}_{i+1,j+1} \times_{M} \mathcal{L}_{j+1, k+1}$ (for all $(i,j,k) \in \mathbb{Z}^3$). The product on this submanifold is given by the pairing of dual line bundles. For example,
\[
(L^{\otimes i+1} \otimes (L^{\otimes j+1})^{*} )\otimes (L^{\otimes j+1} \otimes (L^{\otimes k+1})^{*} ) \to (L^{\otimes i+1} \otimes (L^{\otimes k+1})^{*} ).
\]
It is now straightforward to check that $\hat{\Phi}$ is multiplicative. 

\item Finally, we define the structure of a prequantized holomorphic symplectic groupoid on the disjoint union 
\[
\bigsqcup_{i \in \mathbb{Z}} (Z_{i,0}, \Omega_{i,0}, U_{i,0}, D_{i,0}). 
\]
The source map is defined as before, and is a holomorphic Poisson morphism to $-\sigma_{C}$. We modify the target map as follows. On $Z_{i,0}$ we define the new target map to be $\varphi_{1}^{i} \circ t$. Since $t$ is holomorphic Poisson to $\sigma_{i}$, the new target map is holomorphic Poisson to $\sigma_{0}$. We represent a point of $Z_{i,0}$ as $(i, g)$. The product on this groupoid is defined by
\[
\tilde{m}( (j,g), (i,h) ) = (j+ i, m(\Phi^{-i}(g), h)). 
\]
It is straightforward to check, using the properties established above of $\mathcal{Z}_{C,\mathbb{Z}}$ and $\Phi$, that this defines the structure of a holomorphic symplectic groupoid over $(M, \sigma_{C})$. A similar formula involving $\hat{\Phi}$ defines a multiplicative prequantization. Finally, the submanifolds $\mathcal{L}_{i,0}$ satisfy $\mathcal{L}_{i,0} \ast \mathcal{L}_{j,0} = \mathcal{L}_{i+j,0}$ with respect to the new multiplication. 
\end{enumerate}

\subsection{Comparison with reduction construction}
In this section we construct an isomorphism between the groupoid constructed in~\ref{appmagdef} and the one constructed in Theorem \ref{deformedbrane}. Recall that the latter groupoid has the decomposition 
\[
(\mathcal{Z}, \Omega) := \sqcup_{n \in \mathbb{\mathbb{Z}}} (Z_{n}, \Omega_{n}),
\]
and each component contains a bisection $\mathcal{L}_{n} \subset Z_{n}$, which we view as a section of the source map. The groupoid in `degree 0' is isomorphic to the symplectic groupoid of $\sigma_{C}$ 
\[
(Z_{0}, \Omega_{0}) \cong (\mathcal{G}_{C}, \Omega_{0}).
\]
Therefore, the groupoid multiplication defines an action of $ (\mathcal{G}_{C}, \Omega_{0})$ on each $(Z_{n}, \Omega_{n})$. We now use this action and the bisection $\mathcal{L}_{n}$ to define a diffeomorphism 
\[
T_{n} : \mathcal{G}_{C} \to Z_{n}, \qquad g \mapsto m(\mathcal{L}_{n}(t(g)), g). 
\]

\begin{lemma}
The map $T_{n}$ satisfies the following equations
\[
t \circ T_{n} = \varphi_{1}^n \circ t, \qquad s \circ T_{n} = s, \qquad T_{n}^{*}\Omega_{n} = \Omega_{0} + t^{*}F_{n}. 
\]
\end{lemma}
\begin{proof}
The formulas follow from Theorem \ref{deformedbrane}. First, for $g \in \mathcal{G}_{C}$ we have 
\begin{align*}
s(T_{n}(g)) &= s(m(\mathcal{L}_{n}(t(g)), g)) = s(g), \\
 t(T_{n}(g)) &= t(m(\mathcal{L}_{n} (t(g)), g)) = (t\circ \mathcal{L}_{n}) (t(g)) = \varphi_{1}^n (t(g)). 
\end{align*}
For the formula involving the symplectic form, we use the fact that the graph of the action is a holomorphic Lagrangian submanifold of 
\[
(Z_{n}, -\Omega_{n}) \times (\mathcal{G}_{C}, -\Omega_{0}) \times (Z_{n}, \Omega_{n}).
\]
Consider the map from $\mathcal{G}_{C}$ into the graph given by 
\[
g \mapsto (\mathcal{L}_{n} (t(g)), g, T_{n}(g)).
\]
The symplectic form pulls back to $0$, and hence 
\[
T_{n}^{*}\Omega_{n} = \Omega_{0} + t^{*}\mathcal{L}_{n}^{*}\Omega_{n} = \Omega_{0} + t^{*}F_{n}.
\]
\end{proof}
As a result of this lemma, $T_{n}$ is a holomorphic symplectic isomorphism $Z_{n,0} \to Z_{n}$ which intertwines the source and target maps. Putting all the maps together, we obtain the map
\[
T : \bigsqcup_{i \in \mathbb{Z}} (Z_{i,0}, \Omega_{i,0}) \to (\mathcal{Z}, \Omega). 
\]
\begin{proposition}
The map $T$ is a holomorphic symplectic groupoid isomorphism. 
\end{proposition}
\begin{proof}
It suffices to check that $T$ is multiplicative. Define a map $\Psi_{n} : \mathcal{G}_{C} \to \mathcal{G}_{C}$ by the equation 
\[
m(\mathcal{L}_{n}(t(g)), g) = m(\Psi_{n}(g), \mathcal{L}_{n}(s(g))). 
\]
Since $\mathcal{L}_{n}$ is a bisection, this map is a groupoid automorphism, since $\mathcal{L}_{n}$ is Lagrangian with respect to $\Im(\Omega_{n})$, $\Psi_{n}$ is a symplectomorphism of $(\mathcal{G}_{C}, \Im(\Omega_{0}))$, and finally, because $t \circ \mathcal{L}_{n} = \varphi_{1}^{n}$, the automorphism $\Psi_{n}$ covers $\varphi_{1}^n$. But this implies that we must have $\Psi_{n} = \Phi^{n}$, where $\Phi$ is the automorphism considered in Section \ref{appmagdef}. Stated otherwise, the bisection $\mathcal{L}_{n}$ satisfies 
\[
m(\mathcal{L}_{n}(t(g)), g) = m(\Phi^n(g), \mathcal{L}_{n}(s(g))). 
\]
Now let $g \in Z_{i,0}$ and $h \in Z_{j,0}$ satisfy $s(g) = \varphi_{1}^{j}\circ t(h)$. Using the above equation, along with the multiplicativity of the $\mathcal{L}_{n}$, we obtain (suppressing the multiplication from the notation in the calculation)
\begin{align*}
m(T_{i}(g), T_{j}(h)) &= \big(\mathcal{L}_{i} (t(g)) \ g\big) \big(\mathcal{L}_{j}(t(h)) \ h\big) \\
&= \mathcal{L}_{i} (t(g)) \ \big(\mathcal{L}_{j}(t (\Phi^{-j}(g))) \ \Phi^{-j}(g) \big) \  h \\
&= \mathcal{L}_{i + j} (t (\Phi^{-j}(g))) \ \big(\Phi^{-j}(g) h \big) \\ 
&= T_{i+j}(\tilde{m}(g,h)). 
\end{align*}
\end{proof}

We now lift $T$ to an isomorphism $\hat{T}$ between the multiplicative prequantizations. First, to construct the lift $\hat{T}_{n}$ of $T_{n}$ consider the following map which factors through the graph of the multiplication
\[
\mathcal{G}_{C} \to Z_{n} \times_{M} \mathcal{G}_{C} \to Z_{n} \times \mathcal{G}_{C} \times Z_{n}, \qquad g \mapsto (\mathcal{L}_{n} \circ t(g), g, T_{n}(g)).
\]
Pulling back the bundle $(U_{n}, D_{n})^{*} \boxtimes (U_{0}, D_{0})^{*} \boxtimes (U_{n}, D_{n})$ we obtain the following bundle over $\mathcal{G}_{C}$:
\[
t^{*}\mathcal{L}_{n}^{*}(U_{n}, D_{n})^{*} \otimes (U_{0}, D_{0})^{*} \otimes T_{n}^{*}(U_{n}, D_{n}). 
\]
The multiplicative cocycle $\Theta_{C}$, which is defined over $Z_{n} \times_{M} \mathcal{G}_{C}$, may be pulled back to $\mathcal{G}_{C}$, yielding a flat isomorphism
\[
\Theta_{C,n}: t^{*}\mathcal{L}_{n}^{*}(U_{n}, D_{n}) \otimes (U_{0}, D_{0}) \to T_{n}^{*}(U_{n}, D_{n}). 
\]
Finally, combining this with the isomorphism constructed in Theorem \ref{deformedbrane}, we arrive at our definition of $\hat{T}_{n}$:
\begin{equation}\label{thatfactor}
\hat{T}_{n} = \Theta_{C,n} \circ \big(t^{*}( \mathcal{L}_{n}^{*}(\phi_{n})^{-1} \circ e^{-R_{n}}) \otimes id \big) : t^{*}(L^{n}, \overline{\nabla^{\otimes n}}) \otimes (U_{0}, D_{0}) \to T_{n}^{*}(U_{n}, D_{n}). 
\end{equation}
Putting all the maps together, we get 
\[
\hat{T} : \bigsqcup_{i \in \mathbb{Z}} (U_{i,0}, D_{i,0}) \to (\mathcal{U}, D). 
\]
In the following, we will find it useful to decompose $\Theta_{C}$ as the a product of two contributions 
\[
\Theta_{C} = e^{i \pi C(J(x), J(y))} E. 
\]
\begin{lemma}
The multiplication $E$ defines a (non-flat) multiplicative cocycle, which satisfies the following equation involving the maps $\phi_{n}$ of Theorem \ref{deformedbrane}
\begin{equation} \label{productequivlemma}
e^{-CJ_{j}(y)}\phi_{i}|_{x} \otimes \phi_{j}|_{y} = \phi_{i+j}|_{m(x,y)} E(x,y),
\end{equation} 
where $x \in Z_{i}$ and $y \in Z_{j}$. 
\end{lemma}
\begin{proof}
The multiplicative structure $E$ is obtained by reducing the `trivial' multiplicative structure on $T^{*}P_{L}$, and $\phi_{n}$ are obtained by reducing the morphisms on $T^{*}P_{L}$ obtained from the canonical sections $s_{n}$ (see the proofs of Theorems \ref{reductionprop} and \ref{deformedbrane}). Working with these structures on $T^{*}P_{L}$ the result follows. 
\end{proof}

\begin{proposition}
The map $\hat{T}$ is an isomorphism of multiplicative line bundles with connection. 
\end{proposition}
\begin{proof}
It suffices to check that $\hat{T}$ is multiplicative, and this may be checked along the submanifolds $\mathcal{L}_{i,0} \times_{M} \mathcal{L}_{j,0}$. Since $\hat{T}$ factors as the composition of two maps~\eqref{thatfactor}, and so we check the multiplicativity of each separately. 

The multiplicativity of the term involving $\Theta_{C,n}$ follows from the associativity of the product (i.e. $\delta \Theta_{C} = 1$). More precisely, the induced product on $t^{*}\mathcal{L}_{n}^{*}(U_{n}, D_{n}) \otimes (U_{0}, D_{0})$ is given by two applications of $\Theta_{C}$. Choosing a point $x \in M$, so that $(\mathcal{L}_{i,0}(\varphi_{j}(x)), \mathcal{L}_{j,0}(x)) \in \mathcal{L}_{i,0} \times_{M} \mathcal{L}_{j,0}$, we may write this product as follows:
\[
\Theta_{C}\otimes \Theta_C : (U_{i}|_{\mathcal{L}_{i}\varphi_{j}(x)} \otimes U_{0}|_{\mathcal{L}_{0}\varphi_{j}(x)}) \otimes  (U_{j}|_{\mathcal{L}_{j}(x)} \otimes U_{0}|_{\mathcal{L}_{0}(x)}) \to (U_{i+j}|_{\mathcal{L}_{i+j}(x)} \otimes U_{0}|_{\mathcal{L}_{0}(x)}).
\]
Now note that $U_{0}$ is the trivial bundle with the trivial equivariant structure. Hence, for $a \in U_{i}|_{\mathcal{L}_{i}\varphi_{j}(x)}$ and $b \in  U_{0}|_{\mathcal{L}_{0}\varphi_{j}(x)}$ we get
\begin{align*}
\Theta_{C}(a \otimes b) &= e^{i \pi C(J_{i}\mathcal{L}_{i}\varphi_{j}(x), J_{0}\mathcal{L}_{0}\varphi_{j}(x))}  \Theta(e^{- \frac{1}{2}CJ_{0}\mathcal{L}_{0}\varphi_{j}(x)}  a, e^{\frac{1}{2}CJ_{i}\mathcal{L}_{i}\varphi_{j}(x)}  b) \\
&= \Theta(a, e^{\frac{1}{2}CJ_{i}\mathcal{L}_{i}\varphi_{j}(x)}  b) \\
&= \Theta(a, b) \\
&= ab.
\end{align*}
The second line follows because $\mathcal{L}_{0}^{*}J_{0} = 0$, the third line follows because the equivariant structure on $U_{0}$ is trivial, and the final line follows because $\Theta$ is a tensor product, which reduces to scalar multiplication when one of the terms if trivial. As a result of this, the product reduces to 
\[
\Theta_{C} : U_{i}|_{\mathcal{L}_{i}\varphi_{j}(x)} \otimes U_{j}|_{\mathcal{L}_{j}(x)} \to U_{i+j}|_{\mathcal{L}_{i+j}(x)}.
\]

The product on $U_{n,0}$, restricted to the point $(\mathcal{L}_{i,0}(\varphi_{j}(x)), \mathcal{L}_{j,0}(x))$, is given by 
\[
\hat{\varphi}_{-j} \otimes id : L^{\otimes i}|_{\varphi_{j}(x)} \otimes L^{j}|_{x} \to L^{\otimes (i+j)}|_{x}. 
\]
Therefore, we need to check that the following diagram commutes. 
\[
\begin{tikzpicture}[scale=2.5]
\node (A) at (0,0) {$L^{\otimes i}|_{\varphi_{j}(x)} \otimes L^{j}|_{x}$};
\node (B) at (2,0) {$L^{\otimes (i+j)}|_{x}$};
\node (C) at (0,1) {$U_{i}|_{\mathcal{L}_{i}\varphi_{j}(x)} \otimes U_{j}|_{\mathcal{L}_{j}(x)}$};
\node (D) at (2,1) {$U_{i+j}|_{\mathcal{L}_{i+j}(x)}$};
\path[->,>=angle 90]
(C) edge node[left]{$(e^{R_{i}} \phi_{i}) \otimes (e^{R_{j}} \phi_{j})$} (A)
(A) edge node[above]{$\hat{\varphi}_{-j} \otimes id$} (B)
(C) edge node[above]{$e^{i \pi C(J,J)} E$} (D)
(D) edge node[right]{$e^{R_{i+j}}\phi_{i+j}$} (B);
\end{tikzpicture}
\]
We can decompose the commutativity into two parts. First, in order to match the exponential pre-factors we must show that
\[
R_{i+j}(x) = R_{i}(\varphi_{j}(x)) + R_{j}(x) - i \pi C(J_{i}\mathcal{L}_{i}\varphi_{j}(x), J_{j}\mathcal{L}_{j}(x)). 
\]
Using the expressions for these functions from Theorem \ref{deformedbrane}, the right hand side is given by 
\begin{align*}
i\pi \big( \int_{0}^{i}\int_{0}^{t} C(\varphi_{t+j}^{*}\mu, \varphi_{s+j}^{*}\mu) ds dt + \int_{0}^{j}\int_{0}^{t}C(\varphi_{t}^{*}\mu, \varphi_{s}^{*}\mu)  ds dt  - i^2 C(  \int_{0}^{i} \varphi_{t+j}^{*}\mu dt, \int_{0}^{j} \varphi_{s}^{*}\mu) ds \big),
\end{align*}
which is equal to $R_{i+j}$. 

Second, applying Equation \ref{productequivlemma}, we obtain the equality 
\[
e^{-CJ_{j}\mathcal{L}_{j}(x)}\phi_{i}|_{\mathcal{L}_{i}\varphi_{j}(x)} \otimes \phi_{j}|_{\mathcal{L}_{j}(x)} = \phi_{i+j}|_{\mathcal{L}_{i+j}(x)} E(\mathcal{L}_{i}\varphi_{j}(x),\mathcal{L}_{j}(x)).
\]
It therefore remains to show that the flow $\hat{\varphi}_{-j}$ is obtained by acting by the element 
\[
\exp(-CJ_{j}\mathcal{L}_{j}(x)) = \exp(iC \int_{0}^{j} \mu\varphi_{s}(x) ds).
\]
This follows from the integral equation defining the flow of $\hat{W}$. 
\end{proof}

 \bibliographystyle{plain}
\bibliography{references}
\end{document}

%% file: preamble.tex
\usepackage[width=16cm,top = 3cm,bottom=3cm]{geometry}
\usepackage{amsthm,amsfonts,amssymb,amsmath,amscd} 
\usepackage[all,cmtip,pdf]{xy}
\usepackage{aliascnt} 
\usepackage{mathrsfs} 
\usepackage{xargs,ifthen} 
\usepackage{comment}
\usepackage{color}
\usepackage{tikz}
\usetikzlibrary{calc}
\usepackage{mathtools}

\makeatletter
\def\@cite#1#2{{\normalfont[{#1\if@tempswa , #2\fi}]}}
\makeatother

\usepackage{hyperref}
\definecolor{tocolor}{rgb}{.1,.1,.1}
\definecolor{urlcolor}{rgb}{.2,.2,.6}
\definecolor{linkcolor}{rgb}{.1,.1,.5}
\definecolor{citecolor}{rgb}{.4,.2,.1}
\hypersetup{colorlinks=true, urlcolor=urlcolor, linkcolor=linkcolor, citecolor=citecolor}

\newcommandx{\thdef}[2]{
	\newaliascnt{#1}{theorem}  
	\newtheorem{#1}[#1]{#2}
	\aliascntresetthe{#1}  
	\newtheorem*{#1*}{#2}
	\expandafter\newcommand\expandafter{\csname #1autorefname\endcsname}{#2}
}

\makeatletter
\newtheorem*{rep@theorem}{\rep@title}
\newcommand{\newreptheorem}[2]{%
\newenvironment{rep#1}[1]{%
 \def\rep@title{#2 \ref{##1}}%
 \begin{rep@theorem}}%
 {\end{rep@theorem}}}
\makeatother

\newtheorem{theorem}{Theorem}[section]
\newreptheorem{theorem}{Theorem}

\thdef{lemma}{Lemma}
\thdef{corollary}{Corollary}
\thdef{conjecture}{Conjecture}
\newreptheorem{conjecture}{Conjecture}
\thdef{proposition}{Proposition}

\theoremstyle{definition}
\thdef{definition}{Definition}
\thdef{notation}{Notation}

\theoremstyle{remark}
\thdef{remark}{Remark}
\thdef{remarks}{Remarks}

\theoremstyle{remark}
\thdef{ex}{Example}

{\hfill $\blacksquare$\end{ex}}

\newtheoremstyle{parag}
  {\topsep}   
  {\topsep}   
  {}  
  {}       
  {\bfseries} 
  {.}         
  { } 
  {}          
\theoremstyle{parag}

\newcommand{\spc}[1]{\mathsf{#1}} 
\newcommand{\shf}[1]{\mathcal{#1}} 

\newcommand{\RR}{\mathbb{R}}
\newcommand{\CC}{\mathbb{C}}

\newcommand{\rbrac}[1]{\left(#1\right)} 

\newcommandx{\fn}[2][2=]{#1\ifthenelse{\equal{#2}{}}{}{\!\rbrac{{#2}}}} 
\newcommandx{\id}[2][2=]{\fn{{\rm id}_{#1}}[#2]} 

\newcommand{\ext}[2][\bullet]{\spc{\Lambda}^{#1}{#2}} 
\newcommandx{\End}[2][1=]{\fn{\spc{End}_{#1}}[#2]} 
\newcommandx{\Hom}[2][1=]{\fn{\spc{Hom}_{#1}}[#2]} 
\newcommandx{\Aut}[2][1=]{\fn{\spc{Aut}_{#1}}[#2]} 
\newcommandx{\image}[1]{\fn{\spc{img}}[#1]} 
\renewcommandx{\ker}[1]{\fn{\spc{ker}}[#1]} 
\newcommandx{\rank}[1]{\fn{\mathrm{rank}}[#1]} 

\newcommandx{\ann}[1]{\fn{\spc{ann}}[#1]} 

\newcommandx{\hlgy}[3][1=\bullet,3=]{\spc{H}_{#1}^{#3}\!\rbrac{{#2}}} 
\newcommandx{\cohlgy}[3][1=\bullet,3=]{\spc{H}^{#1}_{#3}\!\rbrac{{#2}}} 
\newcommandx{\chow}[3][1=\bullet,3=]{\spc{A}^{#1}_{#3}\!\rbrac{{#2}}} 

\newcommandx{\Ext}[3][1=\bullet,3=]{\fn{\spc{Ext}^{#1}_{#3}}[{#2}]} 
\newcommandx{\Tor}[3][1=\bullet,3=]{\fn{\spc{Tor}^{#1}_{#3}}[{#2}]} 

\newcommandx{\Pic}[1]{\fn{\spc{Pic}}[{#1}]} 
\newcommandx{\chernalg}[2][1=\bullet]{\fn{\spc{Chern}^{#1}}[{#2}]} 
\newcommandx{\chern}[2][1=]{\fn{c_{#1}}[#2]} 
\newcommandx{\ch}[2][1=]{\fn{\mathrm{ch}_{#1}}[{#2}]} 

\newcommandx{\sKer}[2][1=]{ \fn{ \shf{K}er_{#1}}[{#2}] } 
\newcommandx{\sHom}[2][1=]{ \fn{ \shf{H}om_{#1}}[{#2}] } 
\newcommandx{\sEnd}[2][1=]{ \fn{ \shf{E}nd_{#1}}[{#2}] } 
\newcommandx{\sExt}[3][1=\bullet,3=]{\fn{\shf{E}xt^{#1}_{#3}}[{#2}]} 
\newcommandx{\sTor}[3][1=\bullet,3=]{\fn{\shf{T}or^{#1}_{#3}}[{#2}]} 

\newcommandx{\forms}[2][1=\bullet]{\Omega^{#1}_{#2}} 
\newcommandx{\can}[1][1=]{\omega_{#1}} 
\newcommandx{\acan}[1][1=]{\omega_{#1}^{-1}} 
\newcommandx{\tshf}[1]{\shf{T}_{#1}} 
\newcommandx{\mvect}[2][1=\bullet]{ \ext[#1]{\tshf{#2}} }
\newcommandx{\der}[2][1=\bullet]{\mathscr{X}^{#1}_{#2}} 
\newcommandx{\sJet}[3][1=,2=]{\shf{J}^{#1}_{#2}#3} 

\newcommandx{\tb}[2][1=]{\spc{T}_{\!#1}{#2}} 
\newcommandx{\ctb}[2][1=]{\spc{T}_{\!#1}^*{#2}} 
\newcommandx{\lie}[2][2=]{\fn{\mathscr{L}_{#1}}[#2]} 
\newcommandx{\hook}[2][2=]{\fn{i_{#1}}[#2]} 

\makeatletter
\newcommand{\thickbar}{\mathpalette\@thickbar}
\newcommand{\@thickbar}[2]{{#1\mkern1.5mu\vbox{
  \sbox\z@{$#1\mkern-1mu#2\mkern-1mu$}%
  \sbox\tw@{$#1\overline{#2}$}%
  \dimen@=\dimexpr\ht\tw@-\ht\z@-.6\p@\relax
  \hrule\@height.4\p@ 
  \vskip1\p@
  \hrule\@height.4\p@ 
  \vskip\dimen@
  \box\z@}\mkern1.5mu}
}
\makeatother